\documentclass[12pt]{amsart}

\usepackage{amscd,amssymb,amsmath,amsthm,mathrsfs}
\usepackage{amsfonts}
\usepackage{fullpage}
\usepackage{enumitem} %for no indent enumerate / itemize
\usepackage{hyperref}

\usepackage{color}
\usepackage[backrefs, alphabetic]{amsrefs}
\usepackage[all]{xy} %% MUST COME AFTER AMSREFS FOR STUPID REASONS.

\newtheorem{maintheorem}{Theorem}

\newtheorem{proposition}{Proposition}[section]
\newtheorem{theorem}[proposition]{Theorem}
\newtheorem{lemma}[proposition]{Lemma}
\newtheorem{fact}[proposition]{Fact}

\newtheorem*{claim*}{Claim}
\newtheorem{conjecture}{Conjecture}

\theoremstyle{definition}

\newtheorem{remark}[proposition]{Remark}

\numberwithin{equation}{section}

\DeclareMathOperator{\cl}{cl}
\DeclareMathOperator{\Aut}{Aut}

\DeclareMathOperator{\Hom}{Hom}
\DeclareMathOperator{\Gal}{Gal}

\DeclareMathOperator{\Spec}{Spec}

\DeclareMathOperator{\Char}{char}

\DeclareMathOperator{\Sym}{Sym}
\let\H\relax
\DeclareMathOperator{\H}{H}

\DeclareMathOperator{\trdeg}{tr.deg}
\DeclareMathOperator{\Isom}{Isom}

\DeclareMathOperator{\divv}{div}

\newcommand{\G}{\mathbb{G}}
\newcommand{\Z}{\mathbb{Z}}

\newcommand{\A}{\mathbb{A}}

\newcommand{\Pbb}{\mathbb{P}}

\newcommand{\Gc}{\mathcal{G}}

\newcommand{\Oc}{\mathcal{O}}
\newcommand{\Pc}{\mathcal{P}}

\newcommand{\Lc}{\mathcal{L}_{\rm cl}}

\newcommand{\Gf}{\mathfrak{G}}
\newcommand{\Mf}{\mathfrak{M}}

\newcommand{\Uf}{\mathfrak{U}}
\newcommand{\Df}{\mathfrak{D}}
\newcommand{\mf}{\mathfrak{m}}

\newcommand{\Cf}{\mathfrak{C}}
\newcommand{\kf}{\mathfrak{k}}
\newcommand{\Kf}{\mathfrak{K}}
\newcommand{\Lf}{\mathfrak{L}}
\newcommand{\Zf}{\mathfrak{Z}}
\newcommand{\tplx}{\underline{\mathbf{x}}}
\newcommand{\tply}{\underline{\mathbf{y}}}

\newcommand{\tplb}{\underline{\mathbf{b}}}
\newcommand{\tplc}{\underline{\mathbf{c}}}

\newcommand{\tbf}{\mathbf{t}}
\newcommand{\vbf}{\mathbf{v}}
\newcommand{\wbf}{\mathbf{w}}
\newcommand{\xbf}{\mathbf{x}}

\newcommand{\Pbf}{\mathbf{P}}
\newcommand{\Qbf}{\mathbf{Q}}

\newcommand{\dimm}{\dim^{\rm M}}

\newcommand{\Gfor}{\mathfrak{G}^1_{\rm rat}}
\newcommand{\Isomrat}{\Isom^{\rm M}_{\rm rat}}

\newcommand{\K}{\operatorname{K}^{\rm M}}
\let\k\relax
\newcommand{\k}{\operatorname{k}^{\rm M}}
\newcommand{\bK}{\overline{\operatorname{K}}^{\rm M}}
\newcommand{\Rfrat}{\mathfrak{R}_{\rm rat}}

\newcommand{\Kcalm}{\mathcal{K}^{\rm M}_{\rm rat}}
\newcommand{\UIsom}{\underline{\Isom}}
\newcommand{\Isomm}{\Isom^{\rm M}}
\newcommand{\UIsomm}{\underline{\Isom}^{\rm M}}

\begin{document}
\title{Reconstructing function fields from rational quotients of mod-$\ell$ Galois groups}
\author{Adam Topaz}
\thanks{This research was supported by NSF postdoctoral fellowship DMS-1304114.}
\address{Adam Topaz \vskip 0pt
Department of Mathematics \vskip 0pt
University of California, Berkeley \vskip 0pt
970 Evans Hall \#3840 \vskip 0pt
Berkeley, CA. 94720-3840 \vskip 0pt
USA}
\email{atopaz@math.berkeley.edu}
\urladdr{http://math.berkeley.edu/~atopaz}
\date{\today}

\begin{abstract}
In this paper, we develop the main step in the \emph{global theory} for the mod-$\ell$ analogue of Bogomolov's program in birational anabelian geometry for higher-dimensional function fields over algebraically closed fields.
More precisely, we show how to reconstruct a function field $K$ of transcendence degree $\geq 5$ over an algebraically closed field, up-to inseparable extensions, from the mod-$\ell$ abelian-by-central Galois group of $K$ endowed with the collection of mod-$\ell$ rational quotients.
\end{abstract}

\maketitle
\tableofcontents

\section{Introduction}
\label{sec:intro}

In the early 1990's, {\sc Bogomolov} \cite{Bogomolov1991} introduced a program whose ultimate goal is to reconstruct higher-dimensional function fields over algebraically closed fields from their pro-$\ell$ \emph{abelian-by-central} Galois groups.
If successful, this program would go far beyond Grothendieck's birational anabelian geometry, since Bogomolov's program deals with \emph{small} pro-$\ell$ Galois groups and with fields which are purely \emph{geometric} in nature.

In this paper, we consider the \emph{mod-$\ell$ abelian-by-central} variant of Bogomolov's program.
The mod-$\ell$ context is fundamentally different than the pro-$\ell$ context, primarily with respect to the \emph{global theory}, because one can no longer use the \emph{fundamental theorem of projective geometry} (see the detailed discussion below).
The primary goal of this paper is to develop the \emph{global theory} in the mod-$\ell$ context by using the results of {\sc Evans-Hrushovski} \cite{Evans1991}, \cite{Evans1995} and {\sc Gismatullin} \cite{Gismatullin2008}, which stem from the so-called ``group-configuration theorem'' in geometric stability theory. 
The analogy between the fundamental theorem of projective geometry and the work of {\sc Evans-Hrushovski} and {\sc Gismatullin} has been long known to model theorists, and it was also recently noticed in the context of Galois theory by {\sc Bogomolov-Tschinkel} \cite{zbMATH06092073}.
Nevertheless, the present paper seems to be the first to use these results in the context of anabelian geometry in an essential way.
\vskip 10pt

To make Bogomolov's program into a precise conjecture, we begin by introducing some notation.
Throughout the paper we will work with a fixed prime $\ell$.
In the context of profinite groups, we will tacitly consider only continuous maps and closed subgroups.
For a field $F$, denote by $\Gc_F := \Gal(F(\ell)|F)$ the maximal pro-$\ell$ Galois group of $F$.
Also, we will let $F^i$ denote the perfect closure of $F$.

Let $K$ be a function field over an algebraically closed field $k$.
In this case, consider the following two Galois groups which are constructed from $\Gc_K$, with notation/terminology due to {\sc Pop} \cite{Pop2011}:
\begin{enumerate}
	\item $\Pi_K^a := \Gc_K/[\Gc_K,\Gc_K]$, the maximal {\bf pro-$\ell$ abelian} Galois group of $K$.
	\item $\Pi_K^c := \Gc_K/[\Gc_K,[\Gc_K,\Gc_K]]$, the maximal {\bf pro-$\ell$ abelian-by-central} Galois group of $K$.
\end{enumerate}

Now let $L$ be another function field over an algebraically closed field $l$.
We denote by $\Isom^i(K,L)$ the collection of isomorphisms $\phi : K^i \rightarrow L^i$.
Note that any such isomorphism $\phi$ automatically satisfies $\phi k = l$, since $k$ resp. $l$ is the set of multiplicatively divisible elements in $K^i$ resp. $L^i$.
If $\Char k = p > 0$, then the Frobenius automorphism $\operatorname{Frob}_p : K^i \rightarrow K^i$ acts on $\Isom^i(K,L)$ by composition.
In this case we denote by $\Isom^i_F(K,L)$ the orbits of this action on $\Isom^i(K,L)$.
If $\Char k = 0$, then we define $\Isom^i_F(K,L) := \Isom^i(K,L) = \Isom(K,L)$ to keep the notation consistent.

We will denote by $\Delta_K$ the kernel of the central extension $\Pi_K^c \twoheadrightarrow \Pi_K^a$.
Note that both $\Pi_K^a$ and $\Delta_K$ are abelian pro-$\ell$ groups, and we will consider them as (additive) $\Z_\ell$-modules.
For $\sigma,\tau \in \Pi_K^a$, we define $[\sigma,\tau] := \tilde\sigma^{-1}\tilde\tau^{-1}\tilde\sigma\tilde\tau$ where $\tilde\sigma,\tilde\tau \in \Pi_K^c$ are some lifts of $\sigma,\tau$.
Since $\Pi_K^c \twoheadrightarrow \Pi_K^a$ is a central extension, the element $[\sigma,\tau] \in \Delta_K$ doesn't depend on the choice of lifts of $\sigma,\tau$, and it is well known that $[\bullet,\bullet] : \Pi_K^a \times \Pi_K^a \rightarrow \Delta_K$ is $\Z_\ell$-bilinear.

We say that an isomorphism $\phi : \Pi_L^a \rightarrow \Pi_K^a$ is {\bf compatible with $[\bullet,\bullet]$} if there exists some isomorphism $\psi : \Delta_L \rightarrow \Delta_K$ such that $\psi[\sigma,\tau] = [\phi\sigma,\phi\tau]$ for all $\sigma,\tau \in \Pi_L^a$.
We denote by $\Isom^c(\Pi_L^a,\Pi_K^a)$ the collection of isomorphisms $\phi : \Pi_L^a \rightarrow \Pi_K^a$ which are compatible with $[\bullet,\bullet]$.
Note that if $\epsilon \in \Z_\ell^\times$ and $\phi \in \Isom^c(\Pi_L^a,\Pi_K^a)$, then we obtain an induced isomorphism $\epsilon \cdot \phi : \Pi_L^a \rightarrow \Pi_K^a$, defined by
\[ (\epsilon \cdot \phi)(x) = \epsilon \cdot \phi(x) = \phi(\epsilon \cdot x). \]
Since $[\bullet,\bullet]$ is $\Z_\ell$-bilinear, it follows that $\epsilon \cdot \phi$ is also an element of $\Isom^c(\Pi_L^a,\Pi_K^a)$.
Thus, we have a canonical left-multiplication action of $\Z_\ell^\times$ on $\Isom^c(\Pi_L^a,\Pi_K^a)$, and we denote the orbits of this action by $\UIsom^c(\Pi_L^a,\Pi_K^a)$.

Since the inclusion $K \hookrightarrow K^i$ induces an isomorphism $G_{K^i} \rightarrow G_K$ of absolute Galois groups, we obtain a canonical map $\Isom^i(K,L) \rightarrow \Isom^c(\Pi_L^a,\Pi_K^a)$.
If $\Char K \neq \ell$, we therefore also have an induced canonical map 
\[ \Isom^i_F(K,L) \rightarrow \UIsom^c(\Pi_L^a,\Pi_K^a). \]
Using the notation above, {\sc Pop} \cite{Pop2012a} formulated the following conjecture which is a precise ``Isom'' version of Bogomolov's program in birational anabelian geometry -- see the ``Target Result'' in loc.cit.
\begin{conjecture}[{\sc Bogomolov} \cite{Bogomolov1991}, {\sc Pop} \cite{Pop2012a}]
\label{conjecture:bogomolov-pop-pro-ell}
Let $K|k$ and $L|l$ be function fields over algebraically closed fields such that $\Char K \neq \ell$ and $\trdeg(K|k) \geq 2$.
Then the canonical map 
\[\Isom^i_F(K,L) \rightarrow \UIsom^c(\Pi_L^a,\Pi_K^a) \]
is a bijection.
\end{conjecture}
The proof of Conjecture \ref{conjecture:bogomolov-pop-pro-ell} should also give some isomorphism-compatible \emph{group-theoretical recipe} which constructs $K^i|k$ as fields from $\Pi_K^c$ as a pro-$\ell$ group.
Also, note that the validity of Conjecture \ref{conjecture:bogomolov-pop-pro-ell} would imply that $K^i|k \cong L^i|l$ as fields if and only if $\Pi_L^c \cong \Pi_K^c$ as pro-$\ell$ groups.
While this conjecture is far from being proven in full generality, it has been settled in the case where $k$ is the \emph{algebraic closure of a finite field}, by {\sc Bogomolov-Tschinkel} \cite{Bogomolov2008a}, \cite{Bogomolov2011}, and separately by {\sc Pop} \cite{Pop2011}.

\subsection{The Pro-$\ell$ Strategy}
\label{subsec: pro-ell strategy}

For readers' sake, we summarize the general strategy for proving Conjecture \ref{conjecture:bogomolov-pop-pro-ell}, as outlined in the introduction of \cite{Pop2012a}.
We will henceforth refer to this as {\bf The Pro-$\ell$ Strategy}.

For $K|k$ as in Conjecture \ref{conjecture:bogomolov-pop-pro-ell}, we note that $\mu_{\ell^\infty} \subset K$.
After fixing an isomorphism $\Z_\ell(1) \cong \Z_\ell$ of $G_K$-modules, it follows from Kummer theory that one has an isomorphism 
\[ \Hom(\Pi_K^c,\Z_\ell) = \Hom(\Pi_K^a,\Z_\ell) \cong \widehat{K^\times}\]
where $\widehat{K^\times}$ denotes the $\ell$-adic completion of $K^\times$.
Next, note that the $\ell$-adic completion map $K^\times \rightarrow \widehat{K^\times}$ has $k^\times$ as its kernel, so that $K^\times/k^\times$ can be embedded as a subgroup of $\widehat{K^\times}$.
Using this observation, {\bf the pro-$\ell$ strategy} for reconstructing $K$ using $\Pi_K^c$ has three key steps:
\begin{enumerate}[leftmargin=*]
	\item First, determine $K^\times/k^\times$ as a subset of 
	\[ \Hom(\Pi_K^c,\Z_\ell) = \Hom(\Pi_K^a,\Z_\ell) \cong \widehat{K^\times}. \]
	Note that $K^\times/k^\times$ can be considered as the (infinite-dimensional) projective space, since it is the projectivization $\mathcal{P}_k(K)$ of $K$ as a $k$-vector space.
	\item Second, determine the projective lines on $\mathcal{P}_k(K)$.
	Since the multiplicative structure of $K^\times/k^\times$ is ``known'' from step (1), it suffices to determine the projective lines of $\mathcal{P}_k(K)$ which contain $1 \in K^\times/k^\times$.
	\item Third, apply the \emph{Fundamental Theorem of Projective Geometry} (cf. \cite{artin1957geometric} Ch. II.10) to $\mathcal{P}_k(K)$ in order to obtain the structure of $K$ as a vector space over $k$, then use the multiplicative structure of $K^\times/k^\times$ to recover the multiplicative structure of $K$.
\end{enumerate}
As with most other results in anabelian geometry, the strategy to tackle the three key steps above consists of two main parts: \emph{the local theory} and \emph{the global theory}.
Some portions of each part are known over an arbitrary algebraically closed field, but one must restrict to the algebraic closure of a finite field for everything to fit together.
The following is a brief summary of each part.
\vskip 5pt
\noindent\emph{The Pro-$\ell$ Local Theory:} Here the goal is to use $\Pi_K^c$ in order to recover information about decomposition and inertia subgroups of $\Pi_K^a$ associated to \emph{divisorial valuations} of $K|k$ (i.e. valuations arising from Weil-prime-divisors on normal models of $K|k$), as well as their Parshin-chains.
The primary tool in this context is {\sc Bogomolov} and {\sc Tschinkel's} theory of \emph{commuting-liftable pairs} in pro-$\ell$ abelian-by-central Galois groups \cite{Bogomolov2007}.
Presently, the best one can do over an arbitrary algebraically closed field is to recover the decomposition/inertia subgroups of $\Pi_K^a$ associated to \emph{quasi-divisorial valuations}; see \cite{Pop2010} for the details and for the precise definition of quasi-divisorial valuations.
Nevertheless, over the algebraic closure of a finite field, quasi-divisorial valuations are precisely the divisorial valuations.

\vskip 5pt
\noindent\emph{The Pro-$\ell$ Global Theory:} Using the local theory as an input, the global theory tackles the three key steps of the pro-$\ell$ strategy, by further using the collection of \emph{rational quotients} of $\Pi_K^a$.
This main step in the pro-$\ell$ global theory is the Main Theorem of \cite{Pop2012a}.
See loc.cit. for the details about pro-$\ell$ rational quotients, but note also that we define the mod-$\ell$ analogue of a rational quotient in \S\ref{subsec:galois-variant}.
From here, the main problem is to reconstruct these rational quotients using the given group-theoretical data.
This is accomplished over the algebraic closure of a finite field in \cite{Pop2011}.

\vskip 5pt
The present paper considers the \emph{mod-$\ell$ analogue} of the context above (see the precise notation introduced below).
Roughly speaking, this corresponds to replacing $\Pi_K^a$ with
\[ \Gc_K^a := \frac{\Gc_K}{[\Gc_K,\Gc_K] \cdot (\Gc_K)^\ell} = \frac{\Pi_K^a}{(\Pi_K^a)^\ell}. \]
Dually by Kummer theory, this corresponds to replacing $\widehat{K^\times}$ by $K^\times/\ell$.
Namely, the objects in the mod-$\ell$ context are essentially obtained by modding out the corresponding pro-$\ell$ objects by $\ell$-th powers.
In particular, \emph{the mod-$\ell$ context generalizes the pro-$\ell$ context}.

The fundamental problem in the mod-$\ell$ context is that the pro-$\ell$ strategy, using the three key steps above, fails from the very beginning.
This is precisely because $K^\times/\ell$ contains no apparent geometric object on which one can apply results concerning projective geometry over $k$.
The main goal of this paper is to overcome this fundamental difficulty, and to develop a mod-$\ell$ version of the \emph{global theory}.
We now discuss the notation/context for the mod-$\ell$ situation in detail.

\subsection{The mod-$\ell$ analogue}
\label{subsec:mod-ell-analogue}

%This paper deals with the mod-$\ell$ analogue of the context introduced above.
We begin by introducing the mod-$\ell$ analogue of the notation above in order to state the mod-$\ell$ analogue of Conjecture \ref{conjecture:bogomolov-pop-pro-ell}.
We first recall the first two non-trivial terms in the mod-$\ell$ \emph{Zassenhauss} filtration of a pro-$\ell$ group $\Gc$:
\begin{enumerate}
	\item $\Gc^{(2)} := [\Gc,\Gc] \cdot \Gc^\ell$.
	\item If $\ell \neq 2$, then $\Gc^{(3)} := [\Gc,\Gc^{(2)}] \cdot \Gc^\ell$.
	\item If $\ell = 2$, then $\Gc^{(3)} := [\Gc,\Gc^{(2)}] \cdot (\Gc^{(2)})^\ell$.
\end{enumerate}
We define $\Gc^a := \Gc/\Gc^{(2)}$ and $\Gc^c := \Gc/\Gc^{(3)}$, and we note that $\Gc^c \twoheadrightarrow \Gc^a$ is a central extension.

Suppose that $K$ is a function field over an algebraically closed field $k$, and recall that $\Gc_K$ denotes the maximal pro-$\ell$ Galois group of $K$.
In this case, we will denote the kernel of $\Gc_K^c \twoheadrightarrow \Gc_K^a$ by $\Zf_K$, and we note that $\Gc_K^a$ and $\Zf_K$ are both $\ell$-elementary abelian.
We will consider both $\Gc_K^a$ and $\Zf_K$ as (additive) $\Z/\ell$-modules.
In comparison with the pro-$\ell$ context, we note that one has
\[ \frac{\Pi_K^c}{(\Pi_K^c)^{(2)}} = \frac{\Pi_K^a}{(\Pi_K^a)^{(2)}} = \Gc_K^a, \ \ \frac{\Pi_K^c}{(\Pi_K^c)^{(3)}} = \Gc_K^c.\]
Hence, our mod-$\ell$ context indeed generalizes the pro-$\ell$ context.

Similarly to the pro-$\ell$ context, for $\sigma,\tau \in \Gc_K^a$, we define $[\sigma,\tau] := \tilde\sigma^{-1}\tilde\tau^{-1}\tilde\tau\tilde\sigma$ where $\tilde\sigma,\tilde\tau \in \Gc_K^c$ are some lifts of $\sigma,\tau$.
Since $\Gc_K^c \rightarrow \Gc_K^a$ is a central extension, we again see that $[\sigma,\tau] \in \Zf_K$ doesn't depend on the choice of lifts of $\sigma,\tau$, and it is again well-known that $[\bullet,\bullet] : \Gc_K^a \times \Gc_K^a \rightarrow \Zf_K$ is $\Z/\ell$-bilinear.
Finally, we denote by $\Zf_K^0$ the (closed) subgroup of $\Zf_K$ which is generated by elements of the form $[\sigma,\tau]$ as $\sigma,\tau \in \Gc_K^a$ vary.
If $\ell \neq 2$, then it follows from the definition of the mod-$\ell$ Zassenhauss filtration that $\Zf_K^0 = \Zf_K$.
On the other hand, if $\ell = 2$, then $\Zf_K^0$ might be properly contained in $\Zf_K$.

Let $L$ be another function field over an algebraically closed field $l$.
We say that an isomorphism $\phi : \Gc_L^a \rightarrow \Gc_K^a$ is {\bf compatible with $[\bullet,\bullet]$} if there exists some isomorphism $\psi : \Zf_L^0 \rightarrow \Zf_K^0$ such that $\psi[\sigma,\tau] = [\phi\sigma,\phi\tau]$ for all $\sigma,\tau \in \Gc_L^a$. 
We define $\Isom^c(\Gc_L^a,\Gc_K^a)$ to be the collection of isomorphisms $\Gc_L^a \rightarrow \Gc_K^a$ which are compatible with $[\bullet,\bullet]$.
If $\epsilon \in (\Z/\ell)^\times$ and $\phi \in \Isom^c(\Gc_L^a,\Gc_K^a)$, then we obtain an induced isomorphism $\epsilon \cdot \phi : \Gc_L^a \rightarrow \Gc_K^a$, defined by 
\[ (\epsilon \cdot \phi)(x) = \epsilon \cdot \phi(x) = \phi(\epsilon \cdot x). \]
Since $[\bullet,\bullet]$ is $\Z/\ell$-bilinear, it follows that $\epsilon \cdot \phi$ is also an element of $\Isom^c(\Gc_L^a,\Gc_K^a)$.
Thus, we obtain a canonical action of $(\Z/\ell)^\times$ on $\Isom^c(\Gc_L^a,\Gc_K^a)$, and we denote the orbits of this action by $\UIsom^c(\Gc_L^a,\Gc_K^a)$.

Since $K \hookrightarrow K^i$ induces an isomorphism $G_{K^i} \rightarrow G_K$ of absolute Galois groups, we obtain a canonical map $\Isom^i(K,L) \rightarrow \Isom^c(\Gc_L^a,\Gc_K^a)$.
If $\Char K \neq \ell$, we therefore also obtain an induced canonical map 
\[ \Isom^i_F(K,L) \rightarrow \UIsom^c(\Gc_L^a,\Gc_K^a). \]
With this notation, we can formulate the following mod-$\ell$ analogue of Conjecture \ref{conjecture:bogomolov-pop-pro-ell}.

\begin{conjecture}
\label{conjecture:mod-ell}
Let $K|k$ and $L|l$ be function fields over algebraically closed fields such that $\Char K \neq \ell$ and $\trdeg(K|k) \geq 2$.
Then the canonical map 
\[\Isom^i_F(K,L) \rightarrow \UIsom^c(\Gc_L^a,\Gc_K^a) \]
is a bijection.
\end{conjecture}
The proof of Conjecture \ref{conjecture:mod-ell} should also provide an isomorphism-compatible group-theoretical recipe which constructs $K^i|k$ as fields from $\Gc_K^c$ as a pro-$\ell$ group.
Moreover, note that the validity of Conjecture \ref{conjecture:mod-ell} would imply that $K^i|k \cong L^i|l$ as fields if and only if $\Gc_L^c \cong \Gc_K^c$ as pro-$\ell$ groups.
As mentioned above, we note that Conjecture \ref{conjecture:mod-ell} can be used when considering Conjecture \ref{conjecture:bogomolov-pop-pro-ell}, because the canonical map $\Isom^i_F(K,L) \rightarrow \UIsom^c(\Gc_L^a,\Gc_K^a)$ considered in Conjecture \ref{conjecture:mod-ell} factors through the map $\Isom^i_F(K,L) \rightarrow \UIsom^c(\Pi_L^a,\Pi_K^a)$ from Conjecture \ref{conjecture:bogomolov-pop-pro-ell}.
Thus, Conjecture \ref{conjecture:mod-ell} reduces Conjecture \ref{conjecture:bogomolov-pop-pro-ell} to proving that the map
\[ \UIsom^c(\Pi_L^a,\Pi_K^a) \rightarrow \UIsom^c(\Gc_L^a,\Gc_K^a) \]
is injective, which we expect would follow from the methods in this paper.

\subsection{The Mod-$\ell$ Strategy}
\label{subsec: the mod-ell strategy}
Until now, no strategy has even been formulated to tackle Conjecture \ref{conjecture:mod-ell}.
Nevertheless, it is natural to expect that such a mod-$\ell$ strategy might involve two parts -- a \emph{local theory} and a \emph{global theory} -- similarly to the pro-$\ell$ context.
While the mod-$\ell$ local theory is essentially understood, it was the mod-$\ell$ global theory which posed a significant hurdle.
The following describes what is known in the mod-$\ell$ situation.
\vskip 5pt
\noindent\emph{The Mod-$\ell$ Local Theory:}
The mod-$\ell$ version of the local theory is understood just as well as the pro-$\ell$ local theory.
More precisely, {\sc Pop} \cite{Pop2011a} shows, using similar methods to the pro-$\ell$ version in \cite{Pop2010}, that minimized inertia/decomposition subgroups of $\Gc_K^a$ associated to Parshin-chains of quasi-divisorial valuations can be determined from the group-theoretical structure of $\Gc_K^c$.
The main difference between the two approaches is that the pro-$\ell$ version uses Bogomolov and Tschinkel's theory of commuting-liftable pairs \cite{Bogomolov2007}, while the mod-$\ell$ version requires the mod-$\ell$ version of the theory of commuting-liftable pairs which was developed by the author in \cite{Topaz2012c} and/or \cite{TopazProc2013}.
\vskip 5pt
\noindent\emph{The Mod-$\ell$ Global Theory:}
Until now, nothing was known concerning a possible mod-$\ell$ global theory.
This is primarily because there is no apparent ``geometric'' object which can be obtained from the given mod-$\ell$ Galois-theoretical data.
More precisely, an analog of the pro-$\ell$ strategy fails from the very beginning in the mod-$\ell$ context, since there is no apparent subset of $\Hom(\Gc_K^a,\Z/\ell) \cong K^\times/\ell$ on which one can hope to apply results concerning projective geometry over $k$.

\vskip 5pt
Despite this fundamental difficulty, in this paper, we develop a mod-$\ell$ \emph{global theory} for higher-dimensional function fields.
Namely, we prove the mod-$\ell$ analogue of the Main Theorem of \cite{Pop2012a}, for function fields of transcendence degree $\geq 5$.
More precisely, we show how to recover the function field $K|k$, up-to inseparable extensions, from $\Gc_K^c$ together with the \emph{rational quotients} of $\Gc_K^a$, which will be explicitly defined in \S\ref{subsec:galois-variant}.
Thus, our main theorem reduces Conjecture \ref{conjecture:mod-ell} in transcendence degree $\geq 5$ to the problem of determining these rational quotients of $\Gc_K^a$ using $\Gc_K^c$.

Our approach uses a modern \emph{analogue} of the fundamental theorem of projective geometry, which stems from the so-called ``group-configuration theorem'' in geometric stability theory.
Much like the fundamental theorem of projective geometry, which determines a field and a vector space from the associated projective space and its lines, the analogous theorem which we use here determines a field and a field extension using the so-called ``combinatorial geometry'' associated to relative algebraic closure.
This analogue of the fundamental theorem of projective geometry was originally developed by {\sc Evans-Hrushovski} \cite{Evans1991}, \cite{Evans1995} for extensions of algebraically closed fields, and was later generalized to arbitrary extensions of fields by {\sc Gismatullin} \cite{Gismatullin2008}.
Thus, the main strategy in this paper is to construct this combinatorial geometry using the given Galois group $\Gc_K^c$ endowed with the collection of rational quotients of $\Gc_K^a$.

The main reason that our main theorems require the assumption that $\trdeg \geq 5$ is because we use the results of \cite{Gismatullin2008}, which in turn uses the results of \cite{Evans1995} where this assumption originally appeared.
As mentioned above, loc.cit. uses the ``group-configuration theorem'' from geometric stability theory.
The assumption that $\trdeg \geq 5$ is required in loc.cit. for a technical step which determines which group-configurations arise from the multiplicative group resp. the additive group of $K$.
In fact, it is mentioned on the first page of \cite{Evans1995} that similar results are expected to hold true for $\trdeg \geq 3$, although this is still open.

On the other hand, all of the results of the present paper leading up to the proof of our main theorems work when $\trdeg \geq 4$, and most of them actually work for arbitrary transcendence degree; see the results in \S\ref{sec:recovering-the-mod-ell-algebraic-lattice} where we first put an explicit assumption on the transcendence degree.
We only restrict to $\trdeg \geq 5$ when we finally use \cite{Gismatullin2008}.
Similarly to {\sc Evans-Hrushovski} \cite{Evans1995}, we actually expect our results to be true for $\trdeg \geq 3$.

Finally, we note that the similarity between the fundamental theorem of projective geometry and the results of {\sc Evans-Hrushovski} and {\sc Gismatullin} has been long known to model theorists; see e.g. the MathSciNet review for \cite{Evans1995}, which was written by {\sc Pillay}.
This similarity was also recently noticed by {\sc Bogomolov-Tschinkel} \cite{zbMATH06092073}.
However, the work of {\sc Bogomolov-Tschinkel} doesn't use the results of {\sc Evans-Hrushovski} and {\sc Gismatullin} in any meaningful way, since the fundamental theorem of projective geometry always applies in their contexts.
Thus, the present paper seems to be the first to use combinatorial geometries for applications in anabelian geometry in an essential way.

\subsection{Main Theorem (Galois variant)}
\label{subsec:galois-variant}

Let $K$ be a function field over an algebraically closed field $k$ such that $\Char k \neq \ell$.
Suppose that $F$ is a subextension of $K|k$ which is algebraically closed in $K$.
Then by Kummer theory, the associated map of abelian pro-$\ell$ Galois groups $\Gc_K^a \twoheadrightarrow \Gc_F^a$ is surjective.
For a (closed) subgroup $H$ of $\Gc_K^a$, we say that the group-theoretical quotient $\pi_H : \Gc_K^a \twoheadrightarrow \Gc_K^a/H$ is a {\bf rational quotient} of $\Gc_K^a$ if $H$ is the kernel of the surjective map $\Gc_K^a \twoheadrightarrow \Gc_F^a$ associated to some relatively algebraically closed subextension $F$ of $K|k$, such that $\trdeg(F|k) = 1$ and $F = k(t)$ for some $t \in K$.
The collection of rational quotients of $\Gc_K^a$ will be denoted by $\Rfrat(K|k)$.

If $L$ is another function field over an algebraically closed field $l$ such that $\Char l \neq \ell$, we say that an isomorphism $\phi : \Gc_L^a \rightarrow \Gc_K^a$ is {\bf compatible with $\Rfrat$} if $\phi$ induces a bijection $\Rfrat(L|l) \rightarrow \Rfrat(K|k)$.
We denote by $\Isom^c_{\rm rat}(\Gc_L^a,\Gc_K^a)$ the collection of isomorphisms $\phi : \Gc_L^a \rightarrow \Gc_K^a$ which are compatible with $[\bullet,\bullet]$ and with $\Rfrat$.
Thus $\Isom^c_{\rm rat}(\Gc_L^a,\Gc_K^a)$ is a subset of $\Isom^c(\Gc_L^a,\Gc_K^a)$.
Moreover, the action of $(\Z/\ell)^\times$ on $\Isom^c(\Gc_L^a,\Gc_K^a)$ restricts to an action on $\Isom^c_{\rm rat}(\Gc_L^a,\Gc_K^a)$.
We denote by $\UIsom^c_{\rm rat}(\Gc_L^a,\Gc_K^a)$ the orbits of this action on $\Isom^c_{\rm rat}(\Gc_L^a,\Gc_K^a)$.
Furthermore, we note that the image of the canonical map $\Isom^i(K,L) \rightarrow \Isom^c(\Gc_L^a,\Gc_K^a)$ lands in the subset $\Isom^c_{\rm rat}(\Gc_L^a,\Gc_K^a)$.
Thus, we obtain an induced canonical map 
\[ \Isom^i_F(K,L) \rightarrow \UIsom^c_{\rm rat}(\Gc_L^a,\Gc_K^a). \]
We are now prepared to state the main theorem of the paper, which is precisely the mod-$\ell$ analogue of the Main Theorem from \cite{Pop2012a} in transcendence degree $\geq 5$.

\begin{maintheorem}[Galois variant]
\label{thm:Main-Theorem-A}
Let $K|k$ and $L|l$ be function fields over algebraically closed fields such that $\Char K \neq \ell$ and $\trdeg(K|k) \geq 5$.
Then the following hold:
\begin{enumerate}
	\item There is an isomorphism-compatible group-theoretical recipe which constructs $K^i|k$ as fields from $\Gc_K^c$ as a pro-$\ell$ group together with $\Rfrat(K|k)$, the collection of rational quotients of $\Gc_K^a$.
	\item If $\UIsom^c(\Gc_L^a,\Gc_K^a)$ is non-empty, then $\Char L \neq \ell$, and the canonical map 
	\[\Isom^i_F(K,L) \rightarrow \UIsom^c_{\rm rat}(\Gc_L^a,\Gc_K^a)\]
	is a bijection.
	In particular, $K|k$ and $L|l$ are isomorphic, up-to inseparable extensions, if and only if there is an isomorphism $\Gc_L^c \xrightarrow{\cong} \Gc_K^c$ such that the induced isomorphism $\Gc_L^a \xrightarrow{\cong} \Gc_K^a$ is compatible with $\Rfrat$.
\end{enumerate}
\end{maintheorem}

In particular, Theorem \ref{thm:Main-Theorem-A} reduces Conjecture \ref{conjecture:mod-ell}, for a function field $K|k$ of transcendence degree $\geq 5$, to reconstructing the collection $\Rfrat(K|k)$ of mod-$\ell$ rational quotients in a group-theoretical way from $\Gc_K^c$.
Currently, this step remains open.
Nevertheless, we expect that this reconstruction is indeed possible, based on the fact that, \emph{over the algebraic closure of a finite field}, {\sc Pop} \cite{Pop2011} reconstructs the \emph{pro-$\ell$ rational quotients} using $\Pi_K^c$.

\subsection{Main Theorem (Minor variant)}
\label{subsec:milnor-variant}

The majority of the paper is devoted to proving an analogue of Theorem \ref{thm:Main-Theorem-A} which deals with $\k_*(K)$, the mod-$\ell$ Milnor K-ring of $K$, instead of $\Gc_K^c$.
The definition of $\k_*(K)$ is reviewed in \S\ref{sec:milnor-K-thy}, although we recall now that $\k_1(K) = K^\times/\ell$.

It is important to note that results which reconstruct function fields using Milnor K-theory have been previously developed by {\sc Bogomolov-Tschinkel} \cite{MR2537087}.
The main difference between the results of loc.cit. and our mod-$\ell$ context, is that \cite{MR2537087} works with the \emph{Milnor K-ring modulo divisible elements}, denoted by $\bK_*(K)$.
On the other hand, Theorem \ref{thm:Main-Theorem-A-milnor} below deals with the \emph{mod-$\ell$ Milnor K-ring $\k_*(K)$}, which is a fairly small quotient of $\bK_*(K)$; more precisely, one has $\k_*(K) = \bK_*(K)/\ell$.

For $K$ a function field over an algebraically closed field $k$, we note that one has 
\[ \bK_1(K) = K^\times/k^\times.\]
In particular, the set $K^\times/k^\times$ underlying the projective space $\Pc_k(K)$ is given.
With this observation, the strategy in \cite{MR2537087} is to first recover the collection of projective lines, and to ultimately apply the fundamental theorem of projective geometry on $\Pc_k(K)$.
It is therefore unlikely that a similar strategy will work when replacing $\bK_*(K)$ with $\k_*(K)$ as we have done in Theorem \ref{thm:Main-Theorem-A-milnor} below, similarly to the reason that the pro-$\ell$ strategy doesn't apply in the mod-$\ell$ context (see \S\ref{subsec: the mod-ell strategy}).

To state the ``Milnor Variant'' of our main theorem, we first recall that Kummer theory yields a canonical perfect pairing
\[ \Gc_K^a \times \k_1(K) \rightarrow \mu_\ell.\]
In the case where $K$ is a function field over an algebraically closed field $k$ such that $\Char k \neq \ell$, we say that a subgroup $A$ of $\k_1(K)$ is a {\bf rational subgroup} if there exists a subextension $F$ of $K|k$ such that $F$ is algebraically closed in $K$, $\trdeg(F|k) = 1$, $F = k(t)$ for some $t \in K$, and $A$ is the image of the canonical map $F^\times \rightarrow \k_1(K)$.
The collection of all rational subgroups of $\k_1(K)$ will be denoted by $\Gfor(K|k)$.
Furthermore, note that the Kummer pairing above yields a one-to-one correspondence between $\Rfrat(K|K)$ and $\Gfor(K|k)$, since rational quotients of $\Gc_K^a$ are Kummer-dual to rational subgroups of $\k_1(K)$.

If $L$ is another function field over an algebraically closed field $l$ such that $\Char l \neq \ell$, then we say that an isomorphism $\phi : \k_1(K) \rightarrow \k_1(L)$ is {\bf compatible with $\k_2$} if the induced map $\phi^{\otimes 2} : \k_1(K)^{\otimes 2} \rightarrow \k_1(L)^{\otimes 2}$ descends to an isomorphism $\k_2(K) \rightarrow \k_2(L)$.
We denote by $\Isomm(\k_1(K),\k_1(L))$ the collection of isomorphisms $\phi : \k_1(K) \rightarrow \k_1(L)$ which are compatible with $\k_2$.
Since the inclusion $K \hookrightarrow K^i$ induces an isomorphism of mod-$\ell$ Milnor K-rings $\k_*(K) \xrightarrow{\cong} \k_*(K^i)$ (see Fact \ref{fact:inseperable-milnor-isomorphism}), we obtain a canonical map $\Isom^i(K,L) \rightarrow \Isomm(\k_1(K),\k_1(L))$.

We say that an isomorphism $\phi : \k_1(K) \rightarrow \k_1(L)$ is {\bf compatible with $\Gfor$} if $\phi$ induces a bijection $\Gfor(K|k) \rightarrow \Gfor(L|l)$.
We denote by $\Isomm_{\rm rat}(\k_1(K),\k_1(L))$ the collection of isomorphisms $\phi : \k_1(K) \rightarrow \k_1(L)$ which are compatible with $\k_2$ and with $\Gfor$.
In particular, $\Isomm_{\rm rat}(\k_1(K),\k_1(L))$ is a subset of $\Isomm(\k_1(K),\k_1(L))$.
Furthermore, it's easy to see that the image of the canonical map $\Isom^i(K,L) \rightarrow \Isomm(\k_1(K),\k_1(L))$ is contained in $\Isomm_{\rm rat}(\k_1(K),\k_1(L))$.

If $\phi \in \Isomm_{\rm rat}(\k_1(K),\k_1(L))$ and $\epsilon \in (\Z/\ell)^\times$, then we obtain an induced isomorphism $\epsilon \cdot \phi : \k_1(K) \rightarrow \k_1(L)$, defined by 
\[ (\epsilon \cdot \phi)(x) = \phi(x)^\epsilon = \phi(x^\epsilon), \]
where we denote $\k_1(K) = K^\times/\ell$ resp. $\k_1(L) = L^\times/\ell$ multiplicatively.
Since $\k_*(K)$ and $\k_*(L)$ are $\Z/\ell$-algebras, it follows that $\epsilon \cdot \phi$ is also an element of $\Isomm_{\rm rat}(\k_1(K),\k_1(L))$.
Thus, we have a canonical action of $(\Z/\ell)^\times$ on $\Isomm_{\rm rat}(\k_1(K),\k_1(L))$ whose orbits we denote by $\UIsomm_{\rm rat}(\k_1(K),\k_1(L))$.
In particular, we obtain an induced map 
\[ \Isom^i_F(K,L) \rightarrow \UIsomm_{\rm rat}(\k_1(K),\k_1(L)) \]
which is the mod-$\ell$ Milnor K-ring analogue of the map $\Isom^i_F(K,L) \rightarrow \UIsom^c_{\rm rat}(\Gc_L^a,\Gc_K^a)$ considered in Theorem \ref{thm:Main-Theorem-A}.
The main goal of the paper is to prove the following ``Milnor variant'' of Theorem \ref{thm:Main-Theorem-A}.
\begin{maintheorem}[Milnor variant]
\label{thm:Main-Theorem-A-milnor}
Let $K|k$ and $L|l$ be function fields over algebraically closed fields, such that $\Char K \neq \ell$, $\Char L \neq \ell$, and $\trdeg(K|k) \geq 5$.
Then the following hold:
\begin{enumerate}
	\item There is an isomorphism-compatible group-theoretical recipe which constructs $K^i|k$ as fields from the following data $\Kcalm(K|k)$:
	\begin{itemize}
		\item The groups $\k_1(K)$ and $\k_2(K)$.
		\item The multiplication map $\k_1(K) \otimes \k_1(K) \rightarrow \k_2(K)$.
		\item The collection $\Gfor(K|k)$ of rational subgroups of $\k_1(K)$.
	\end{itemize}
	\item The canonical map 
	\[\Isom^i_F(K,L) \rightarrow \UIsomm_{\rm rat}(\k_1(K),\k_1(L)) \]
	is a bijection.
	In particular, $K|k$ and $L|l$ are isomorphic, up-to inseparable extensions, if and only if there is an isomorphism $\k_1(K) \xrightarrow{\cong} \k_1(L)$ which is compatible with $\k_2$ and with $\Gfor$.
\end{enumerate}
\end{maintheorem}

%%%%%%%%%%%%%%%%%%%%%%%%%%%%%%%%%%%%%%%%%%%%%%%%%%%%%

\subsection{A guide through the paper}
\label{sub:guide-through-the-paper}

The following is a description of the various sections of this paper, and a summary of the proofs of the main theorems:
\begin{enumerate}[leftmargin=*]
	\item In \S\ref{sec:milnor-K-thy}, we recall the definition of the Milnor K-ring $\K_*(K)$ of a field $K$ as well as its mod-$\ell$ version $\k_*(K)$.
	We also recall some basic properties of the Milnor K-ring, especially with respect to \emph{tame symbols}.
	\item In \S\ref{sec:milnor-k-thy-of-function-fields}, we prove some vanishing/non-vanishing results in mod-$\ell$ Milnor K-theory of a function field $K$ over an algebraically closed field $k$.
	The ``vanishing'' result states that $\k_s(K) = 0$ whenever $s > \trdeg(K|k) =: d$.
	This follows from the fact that $K$ has $\ell$-cohomological dimension $d$, combined with the Bloch-Kato conjecture, which is now a theorem of {\sc Voevodsky-Rost} et al. \cite{Voevodsky2011}, \cite{Rost-chain-lemma}, \cite{zbMATH05594008}.
	On the other hand, the ``non-vanishing'' results assert that there are ``many'' elements $x_1,\ldots,x_r \in \k_1(K)$ such that $\{x_1,\ldots,x_r\} \neq 0$ in $\k_r(K)$, for $r \leq d$, and that these elements can generally be found in relatively algebraically closed subextensions of $K|k$.
	\item In \S\ref{sec:graded_lattices} we recall the definition of a \emph{graded lattice} and give the main example considered in this paper: the graded lattice $\G^*(K|k)$ associated to relative algebraic closure in the field extension $K|k$.
	In this section we also use the results of \S\ref{sec:milnor-k-thy-of-function-fields} to prove that an isomorphic copy of this graded lattice, denoted $\Gf^*(K|k)$, can be found inside the lattice of subgroups of $K^\times/\ell = \k_1(K)$, in the context of the main theorems.
	\item The key propositions required to prove the main theorems are found in \S\ref{sec:recovering-the-mod-ell-algebraic-lattice}.
	Namely, under the assumptions/notation of our main theorems, we show how to reconstruct $\Gf^*(K|k)$, as a subset of the lattice of subgroups of $\k_1(K)$, using $\k_*(K)$ along with $\Gfor(K|k)$.
	\item In \S\ref{section:combinatorial-geometries}, we recall the standard construction which produces a set with a closure operation from a graded lattice.
	When applied to the graded lattice $\G^*(K|k)$, this construction produces the combinatorial geometry $\G(K|k)$ which was considered in \cite{Evans1991}, \cite{Evans1995}, and \cite{Gismatullin2008}.
	In this section we also recall the main result of loc.cit. which reconstructs $K^i|k$ using the combinatorial geometry $\G(K|k)$; this theorem can be seen as the appropriate analogue of the fundamental theorem of projective geometry which applies in this situation. 
	\item Theorem \ref{thm:Main-Theorem-A-milnor} is proved in \S\ref{section:proof-of-theorem-milnor}.
	Some extra work is required to prove the bijectivity of the map $\Isom^i_F(K,L) \rightarrow \UIsomm_{\rm rat}(\k_1(K),\k_1(L))$ considered in Theorem \ref{thm:Main-Theorem-A-milnor}(2).
	\item Finally, in \S\ref{sec:zassenhauss}, we recall the cohomological framework that allows us to translate between Theorems \ref{thm:Main-Theorem-A} and \ref{thm:Main-Theorem-A-milnor}.
	These cohomological results can be seen as a group-theoretical reformulation of the Merkurjev-Suslin Theorem \cite{Merkurjev1982}.
	We conclude the paper by completing the proof of Theorem \ref{thm:Main-Theorem-A} in \S\ref{sec:proof-of-theorem-galois}.
\end{enumerate}

\subsection*{Acknowledgments}
The author warmly thank Florian Pop and Thomas Scanlon for numerous technical discussions concerning the topics in this paper.
The author also thanks Martin Hils and James Freitag for several helpful discussions.
The manuscript was written during the MSRI semester on Model Theory, Arithmetic Geometry and Number theory in the spring of 2014.
The author thanks MSRI and the organizers of this semester for their hospitality and for an excellent research environment.
The author also thanks the referee for his/her extremely useful comments which helped improve the paper in many ways.

\section{Milnor K-theory}
\label{sec:milnor-K-thy}

Let $K$ be an arbitrary field.
We recall that the {\bf $n$-th Milnor K-group} of $K$ is defined as follows:
\[ \K_n(K) := \frac{(K^\times)^{\otimes n}}{\langle a_1\otimes\cdots\otimes a_n \ : \ \exists i < j, \ a_i + a_j = 1 \rangle}.\]
The tensor product makes $\K_*(K) := \bigoplus_{n \geq 0} \K_n(K)$ into a graded-commutative ring with $\K_0(K) = \Z$, and we denote the product in this ring by $\{\bullet,\ldots,\bullet\}$.

We will use the fairly standard notation $\k_n(K) := \K_n(K)/\ell$ for the {\bf mod-$\ell$ Milnor K-groups} of $K$.
Note that the tensor product also makes $\k_*(K) := \bigoplus_{n \geq 0} \k_n(K)$ into a graded-commutative ring with $\k_0(K) = \Z/\ell$.
We will abuse the notation and denote the product in $\k_*(K)$ also by $\{\bullet,\ldots,\bullet\}$.

It is clear from the definition that both $\K_*(\bullet)$ and $\k_*(\bullet)$ are functors.
For an embedding of fields $\iota : K \hookrightarrow L$, we denote by $\iota_*$ the induced homomorphism $\K_*(K) \rightarrow \K_*(L)$ and/or $\k_*(K) \rightarrow \k_*(L)$.

If $L|K$ is a finite extension, we recall that we have a homomorphism $\operatorname{N}_{L|K} : \K_*(L) \rightarrow \K_*(K)$ of abelian groups called the {\bf norm}, and that this map agrees with the usual field norm $\operatorname{N}_{L|K} : L^\times \rightarrow K^\times$ in degree $1$.
The main property of this map which we will use is that the composition $\operatorname{N}_{L|K} \circ \iota_* : \K_*(K) \rightarrow \K_*(K)$ is precisely multiplication by $[L:K]$ (cf. \cite{Gille2006} Remark 7.3.1).
Using this, we obtain the following fact which will be useful in reducing several of our arguments involving finite extensions to the separable case.
\begin{fact}
\label{fact:inseperable-milnor-isomorphism}
Suppose that $\iota: K \hookrightarrow L$ is a purely inseparable finite extension of fields of characteristic different from $\ell$.
Then the canonical map $\iota_* : \k_*(K) \rightarrow \k_*(L)$ is an isomorphism.
\end{fact}
\begin{proof}
Let $p := \Char K$ and note that $\ell$ doesn't divide $[L:K]$.
From this we see that $\iota_*$ is injective, since $\operatorname{N}_{L|K} \circ \iota_*$ is multiplication by $[L:K]$ which is invertible in $\Z/\ell$.
Moreover, since $L|K$ is finite and purely inseparable, there exists some $m \geq 0$ such that $L \subset K^{1/p^m}$.
If $\eta \in \ker(\k_*(L) \rightarrow \k_*(K^{1/p^m}))$, then there exists a finite subextension $L'$ of $K^{1/p^m}|L$ such that $\eta \in \ker(\k_*(L) \rightarrow \k_*(L'))$.
Since $L'|L$ is purely inseparable, the argument above implies that $\eta = 0$.
Namely, the canonical map $\k_*(L) \rightarrow \k_*(K^{1/p^m})$ is injective as well.

On the other hand, the canonical map $\k_*(K) \rightarrow \k_*(K^{1/p^m})$ is an \emph{isomorphism} since $p$ is invertible in $\Z/\ell$.
Finally, the inclusion $K \hookrightarrow K^{1/p^m}$ factors through $L \hookrightarrow K^{1/p^m}$, and it follows that the map $\iota_* : \k_*(K) \rightarrow \k_*(L)$ is surjective.
We conclude that $\iota_* : \k_*(K) \rightarrow \k_*(L)$ is indeed an isomorphism.
\end{proof}

\subsection{Milnor dimension} 
\label{sub:milnor_dimension}

Let $K$ be an arbitrary field and let $S$ be a subset of $\k_1(K)$.
Define the {\bf mod-$\ell$ Milnor dimension} of $S$ in $\k_*(K)$, denoted $\dimm_K(S)$, to be the unique element of $\{0,1,\ldots,\infty\}$ which satisfies the following property: For all positive integers $d$, one has $\dimm_K(S) \geq d$ if and only if there exist $s_1,\ldots,s_d \in S$ such that $\{s_1,\ldots,s_d\} \neq 0$ in $\k_d(K)$.

For a non-negative integer $r \geq 0$, we define $\Mf^r(K)$ to be the (possibly empty) collection of subgroups of $\k_1(K)$ which are \emph{maximal} among subgroups $B$ of $\k_1(K)$ such that $\dimm_K(B) = r$.
In particular, the distinct elements of $\Mf^r(K)$ are non-comparable as subgroups of $\k_1(K)$.

By Zorn's lemma, if there exists some subgroup $A_0$ of $\k_1(K)$ such that $\dimm_K(A_0) = r$, then $\Mf^r(K)$ is non-empty.
Indeed, suppose that $(A_i)_i$ is a chain of subgroups of $\k_1(K)$ such that $\dimm_K(A_i) = r$ for all $i$, and let $A = \bigcup_i A_i$ be the union of the $A_i$ in $\k_1(K)$.
Then it easily follows from the definition that $A$ is also a subgroup of $\k_1(K)$ such that $\dimm_K(A) = r$.

\subsection{Valued fields and tame symbols}
\label{subsec:tame-symbols}

Let $(K,v)$ be a valued field.
We will denote the associated value group by $vK$, the residue field by $Kv$, the valuation ring by $\Oc_v$, and the valuation ideal by $\mf_v$.
In certain situations where we wish to keep $K$ out of the notation, we will use the notation $\Gamma_v$ for the value group of $v$ and/or $\kappa(v)$ for the residue field.
Moreover, we will generally denote the canonical map $\Oc_v \rightarrow Kv = \kappa(v)$ by $x \mapsto \bar x$.
Lastly, the group $\Oc_v^\times$ of $v$-units will be denoted by $U_v$.

Recall that $(K,v)$ is called a {\bf discretely valued field} (and/or $v$ is a {\bf discrete valuation}) provided that $v$ is a valuation of $K$ such that $vK = \Z$.
If $(K,v)$ is such a discretely valued field, we obtain a homomorphism $\{\bullet\}_v : \K_{n+1}(K) \rightarrow \K_n(Kv)$ for each $n$, called the {\bf tame symbol associated to $v$}, which is uniquely defined by the condition (cf. \cite{Gille2006} Proposition 7.1.4):
\begin{align}
\label{align:tame-definition}
 \{\pi,u_1,\ldots,u_n\}_v = \{\bar u_1,\ldots,\bar u_n\}
\end{align}
for any element $\pi \in K^\times$ such that $v(\pi) = 1$, and $v$-units $u_1,\ldots,u_n \in U_v$.
In particular, note that the tame symbol $K^\times = \K_1(K) \rightarrow \K_0(Kv) = \Z$ in degree $1$ is precisely the homomorphism $v : K^\times \rightarrow \Z$.
Of course, the tame symbol induces homomorphisms $\k_{n+1}(K) \rightarrow \k_n(Kv)$ on mod-$\ell$ Milnor K-theory, which we also denote by $\{\bullet\}_v$.

The tame symbol behaves functorially in field extensions, up-to multiplication by the ramification index, as follows.
Suppose that $\iota : (K,v) \hookrightarrow (L,w)$ is a finite extension of discretely valued fields.
Let $e = e(w|v)$ be the ramification index of $w|v$; in other words, $e := [wL:vK] = w(\pi)$ for an element $\pi \in K^\times$ such that $v(\pi) = 1$.
Using equation (\ref{align:tame-definition}), it is easy to see that the following diagram commutes:
\[
\xymatrix{
	\K_{*+1}(K)\ar[d]_{\iota_*} \ar[r]^{\{\bullet\}_v} & \K_*(Kv) \ar[d]^{e \cdot \iota_*} \\
	\K_{*+1}(L) \ar[r]_{\{\bullet\}_w} & \K_*(Lw)
}
\]
By tensoring this diagram with $\Z/\ell$, we obtain a similar commutative diagram for $\k_*$.

\subsection{Parshin-chains}
\label{subsec:parshin-chains}

Let $K$ be a field.
A {\bf Parshin-chain} of length $r$ on $K$ is an ordered collection of valuations $\vbf = (v_1,\ldots,v_r)$ where each term is a valuation on the residue field of the previous one:
\begin{enumerate}
	\item $v_1$ is a valuation on $K$.
	\item $v_{i+1}$ is a valuation on $\kappa(v_i)$ for all $i = 1,\ldots,r-1$.
\end{enumerate}
For a Parshin-chain $\vbf$ as above, we will write $\vbf^\circ := v_r \circ \cdots \circ v_1$ for the valuation theoretic composition of the terms of $\vbf$.
Thus, $\vbf^\circ$ is a valuation on $K$.
The residue field of $\vbf$, denoted $K\vbf$ or $\kappa(\vbf)$, is defined to be $K\vbf^\circ$, the residue field of $\vbf^\circ$.
Note that $K\vbf$ is precisely $\kappa(v_r)$, the residue field of $v_r$ where $r$ is the length of $\vbf$.
Similarly, the group of units of $\vbf$, denoted $U_{\vbf}$, is defined to be $U_{\vbf^\circ}$, the units of $\vbf^\circ$.

Suppose that $\vbf = (v_1,\ldots,v_r)$ is a Parshin-chain of length $r$ on $K$ and let $1 \leq s \leq r$ be given.
Then the restricted collection $\vbf_{\leq s} := (v_1,\ldots,v_s)$ is a Parshin-chain of length $s$ on $K$.
Similarly, the collection $\vbf_{> s} := (v_{s+1},\ldots,v_r)$ is a Parshin-chain of length $r-s$ on $\kappa(\vbf_{\leq s}) = \kappa(v_s)$.
Moreover, we see that $\vbf_{\leq s}^\circ$ is a coarsening of $\vbf^\circ$, and that $\vbf^\circ = \vbf_{> s}^\circ \circ \vbf_{\leq s}^\circ$.

\subsection{Discrete Parshin-chains and tame symbols}
\label{subsec:discrete-parshin-chains}

We will say that a Parshin-chain $\vbf = (v_1,\ldots,v_r)$ on $K$ is {\bf discrete} provided that $\Gamma_{v_i} = \Z$ for all $i = 1,\ldots,r$.
Given $x_1,\ldots,x_r \in K$, we say that $\xbf = (x_1,\ldots,x_r)$ is a {\bf uniformizing system} for a discrete Parshin-chain $\vbf = (v_1,\ldots,v_r)$ on $K$ if $x_1,\ldots,x_r \in K$ satisfy the following (inductive) conditions:
\begin{enumerate}
	\item If $r = 1$, then $\xbf = (x_1)$ is a uniformizing system for $\vbf = (v_1)$ if and only if $v_1(x_1) = 1$.
	I.e. a uniformizing system for a discrete valuation $v$, considered as a length-$1$ Parshin-chain, is simply a uniformizer for $v$.
	\item If $r > 1$, suppose uniformizing systems have been defined for discrete Parshin-chains of length $r-1$.
	Then $\xbf = (x_1,\ldots,x_r)$ is a uniformizing system for $\vbf = (v_1,\ldots,v_r)$ provided that the following conditions hold: 
	\begin{enumerate}
		\item $x_1$ is a uniformizer for $v_1$.
		\item $x_2,\ldots,x_r \in U_{v_1}$.
		\item $(\bar x_2,\ldots,\bar x_r)$ is a uniformizing system for $(v_2,\ldots,v_r)$, where $\bar x_i$ denotes the image of $x_i$ in $\kappa(v_1)$.
	\end{enumerate}
\end{enumerate}
It is easy to see that uniformizing systems always exist for any discrete Parshin-chain $\vbf = (v_1,\ldots,v_r)$ of any length.
Moreover, if $\xbf = (x_1,\ldots,x_r)$ is a uniformizing system for $\vbf = (v_1,\ldots,v_r)$, and $1 \leq s \leq r$, then $(x_1,\ldots,x_s)$ is a uniformizing system for $\vbf_{\leq s}  = (v_1,\ldots,v_s)$, and $(\bar x_{s+1},\ldots,\bar x_r)$ is a uniformizing system for $\vbf_{> s} = (v_{s+1},\ldots,v_r)$, where $\bar x_i$ denotes the image of $x_i$ in $\kappa(\vbf_{\leq s}) = \kappa(v_s)$.

For a discrete Parshin-chain $\vbf = (v_1,\ldots,v_r)$, we may compose the tame symbols associated to each $v_i$ and obtain the {\bf Tame symbol associated to $\vbf$}:
\[\{\bullet\}_\vbf := \{ \cdots\{ \{ \bullet \}_{v_1} \}_{v_2}\cdots \}_{v_r}: \K_{*+r}(K) \rightarrow \K_*(K\vbf). \]
Similarly, we obtain induced homomorphisms on mod-$\ell$ Milnor K-groups $\{\bullet\}_\vbf  : \k_{*+r}(K) \rightarrow \k_*(K\vbf)$.

\subsection{Prolongations of Parshin-chains}
\label{subsec:prolongations-of-parshin-chains}

Suppose that $K \hookrightarrow L$ is an extension of fields and $\vbf = (v_1,\ldots,v_r)$ is a Parshin-chain on $K$.
We say that a Parshin-chain $\wbf$ on $L$ is a {\bf prolongation} of $\vbf$ provided that the following (inductive) conditions hold:
\begin{enumerate}
	\item The length of $\vbf$ is the same as the length of $\wbf$; say both have length $r$, and write $\vbf = (v_1,\ldots,v_r)$ and $\wbf = (w_1,\ldots,w_r)$.
	\item If $r = 1$, then $\wbf = (w_1)$ is a prolongation of $\vbf = (v_1)$ if and only if the valuation $w_1$ is a prolongation of $v_1$ to $L$.
	\item If $r > 1$, suppose prolongations have been defined for Parshin-chains of length $r-1$.
	Then $\wbf = (w_1,\ldots,w_r)$ is a prolongation of $\vbf = (v_1,\ldots,v_r)$ to $L$ provided that the following conditions hold:
	\begin{enumerate}
		\item $w_1$ is a prolongation of $v_1$ to $L$.
		In particular this implies that $L w_1$ is a field extension of $K v_1$.
		\item $(w_2,\ldots,w_r)$ is a prolongation of $(v_2,\ldots,v_r)$ to $L w_1$. 
	\end{enumerate}
\end{enumerate}
Since valuations can be prolonged to any field extension, the same is true for Parshin-chains.

Suppose that $\vbf = (v_1,\ldots,v_r)$ is a \emph{discrete} Parshin-chain on $K$ and that $L|K$ is a finite extension.
Let $\wbf = (w_1,\ldots,w_r)$ be a prolongation of $\vbf$ to $L$.
It is easy to see in this case that $\wbf$ must be a discrete Parshin-chain on $L$.
In this context, the {\bf ramification indices} of $\wbf|\vbf$, denoted $e(\wbf|\vbf) = (e_1,\ldots,e_r)$, is a sequence of positive integers where $e_i = e(w_i|v_i)$ is the usual ramification index of $w_i|v_i$.

\begin{lemma}
\label{lemma:parshin-tame-symbols}
Let $K$ be a field and let $\vbf = (v_1,\ldots,v_r)$ be a discrete Parshin-chain on $K$.
Let $L|K$ be a (possibly trivial) finite extension and let $\wbf = (w_1,\ldots,w_r)$ be a prolongation of $\vbf$ to $L$ with ramification indices $e(\wbf|\vbf) = (e_1,\ldots,e_r)$.
Let $(x_1,\ldots,x_r)$ be a uniformizing system for $\vbf$.
Then the following hold:
\begin{enumerate}
	\item The following equality holds in $\K_0(L\wbf) = \Z$: \[\{x_1,\ldots,x_r\}_{\wbf} = e_1 \cdots e_r. \]
	\item Let $y_1,\ldots,y_s \in U_\wbf$ be given.
	Then the following equality holds in $\K_s(L\wbf)$: \[\{x_1,\ldots,x_r,y_1,\ldots,y_s\}_{\wbf} = e_1 \cdots e_r \cdot \{\bar y_1,\ldots,\bar y_s\}.\]
\end{enumerate}
\end{lemma}
\begin{proof}
Both (1) and (2) follow immediately by applying Equation (\ref{align:tame-definition}) $r$ times.
\end{proof}

\section{Milnor K-theory of Function Fields}
\label{sec:milnor-k-thy-of-function-fields}

Throughout this section, $k$ will denote an algebraically closed field of characteristic $\neq \ell$.
Recall that $K$ is called a {\bf function field} over $k$ if $K$ is a finitely generated field extension of $k$.
We say that $X$ is a {\bf $k$-variety} if $X$ is an integral separated scheme of finite type over $k$.

For a (scheme-theoretic) point $P$ in an integral $k$-scheme $X$, we denote by $\overline{P}$ the closure of $\{P\}$ in $X$, considered as a reduced (hence integral) closed subscheme of $X$.
Also, we write $(\Oc_{X,P},\mf_P)$ for the local ring at the point $P$, and recall that $k(P) = \Oc_{X,P}/\mf_P$ is the function field of $\overline{P}$. 
If $Y \rightarrow X$ is a morphism of $k$-schemes and $P \in X$, then we write $Y_P := \Spec k(P) \times_X Y$ for the fiber of $Y \rightarrow X$ over $P$.
Finally, we note that if $X$ is a $k$-variety and $P \in X$, then $\overline{P}$ is a $k$-variety as well.

Although we will usually work with $k$-varieties, the following slightly more general notion of a \emph{model} of a function field will be useful for certain definitions/constructions.
If $K|k$ is a function field, we say that $X$ is a {\bf model} for $K|k$ if $X$ is one of the following types of $k$-schemes:
\begin{enumerate}
	\item $X$ is a $k$-variety whose function field is $K$.
	\item $X = \Spec \Oc_{Y,y}$ for some $k$-variety $Y$ whose function field is $K$ and some $y \in Y$.
\end{enumerate}
In particular, a model $X$ for $K|k$ is a $k$-variety if and only if $X$ is of finite type over $k$.

\subsection{Prime-divisors}
\label{subsec:parshin-regular-parameters}

Suppose that $K$ is a function field over $k$ and that $X$ is a model for $K|k$.
A regular point $P$ of codimension 1 in $X$ will be called a {\bf prime-divisor} on $X$.
Recall that a prime-divisor $P$ on $X$ yields a discrete valuation $v_P$ of $K$, whose valuation ring is precisely $\Oc_{X,P}$, the local ring of $P$ in $X$.
We say that a valuation $v$ of $K$ is {\bf divisorial} if $v = v_P$ for some prime-divisor $P$ on some $k$-variety $Y$ such that $k(Y) = K$.

An {\bf $r$-prime divisor} on $X$, denoted $\Pbf = (P_1,\ldots,P_r)$, is an ordered collection of (scheme-theoretic) points $P_i$ of $X$, and is defined inductively as follows:
\begin{enumerate}
	\item If $r = 1$, then $\Pbf = (P)$ is a $1$-prime-divisor on $X$ if and only if $P$ is a prime-divisor on $X$.
	We will not distinguish between prime-divisors $P$ and the associated $1$-prime-divisors $\Pbf = (P)$.
	\item If $r > 1$, then $(P_1,\ldots,P_r)$ is an $r$-prime-divisor on $X$ if and only if $P_1$ is a prime-divisor on $X$ and $(P_2,\ldots,P_r)$ is an $(r-1)$-prime-divisor on $\overline{P_1}$.
\end{enumerate}
Similarly to the way in which a prime divisor yields a discrete valuation, an $r$-prime-divisor $\Pbf = (P_1,\ldots,P_r)$ on $X$ yields a discrete Parshin-chain of length $r$ on $K$, denoted $\vbf_{\Pbf}$. 
Explicitly, the Parshin-chain $\vbf_\Pbf = (v_1,\ldots,v_r)$ is defined as follows:
\begin{enumerate}
	\item The valuation $v_1 = v_{P_1}$ is the valuation of $K$ associated to $P_1$.
	\item For $i = 1,\ldots,r-1$, the valuation $v_{i+1} = v_{P_{i+1}}$ is the valuation of $k(P_{i}) = \kappa(v_{i})$ associated to $P_{i+1}$.
\end{enumerate}
In particular, if $\Pbf = (P_1,\ldots,P_r)$ is an $r$-prime-divisor on $X$ with associated Parshin-chain $\vbf_\Pbf = (v_1,\ldots,v_r)$, then $k(P_r)$, the residue field of the scheme-theoretic point $P_r$ in $X$, is precisely $\kappa(\vbf_\Pbf)$, the residue field of the Parshin-chain $\vbf_\Pbf$.

\subsection{Regular parameters}
\label{sub:regular-parameters}

One key way to obtain an $r$-prime-divisor is using regular coordinates.
Let $X$ be a $k$-variety such that $k(X) = K$ and let $x$ be a \emph{regular closed point} of $X$.
Consider $(A,\mf) := (\Oc_{X,x},\mf_x)$ the regular local $k$-algebra associated to $x \in X$.
Let $x_1,\ldots,x_d$ be a regular system of parameters for $(A,\mf)$.
Fix $r \leq d$ and let $\xbf = (x_1,\ldots,x_r)$; such an $\xbf$ will be called a {\bf partial system of regular parameters} of length $r$ in $A$.
We can associate to such a system $\xbf$ an $r$-prime-divisor $\Pbf_\xbf = (P_1,\ldots,P_r)$ on $\Spec A$, by letting $P_i$ be the generic point of $V(x_1,\ldots,x_i)$ for each $i = 1,\ldots,r$.
Since $P_i$ is a regular point of codimension 1 in $\Spec A/(x_1,\ldots,x_{i-1}) = V(x_1,\ldots,x_{i-1})$ for each $i = 1,\ldots,r$, we see that $\Pbf_\xbf$ is indeed an $r$-prime-divisor on $\Spec A$.

Moreover, because we started with a regular closed point $x$ in $X$, it is easy to see that this construction also yields an $r$-prime-divisor on $X$, by mapping the terms of $\Pbf_\xbf$ via the canonical map $\Spec A \rightarrow X$.
We will abuse the notation and also denote by $\Pbf_\xbf$ the associated $r$-prime-divisor on $X$.
Furthermore, it is clear from the definitions that, setting $\Pbf = \Pbf_\xbf$, the system $\xbf$ is a \emph{uniformizing system} for the associated Parshin-chain $\vbf_{\Pbf}$, as defined in \S\ref{subsec:discrete-parshin-chains}.

\subsection{Prolongations of $r$-prime-divisors}
\label{subsec:prolongations-of-r-prime-divisors}

Suppose that $X$ is a normal $k$-variety with function field $K$.
Let $L$ be a finite extension of $K$ and let $Y$ denote the normalization of $X$ in $L$.
Suppose that $\Pbf = (P_1,\ldots,P_r)$ is an $r$-prime-divisor on $X$.
We say that an $r$-prime-divisor $\Qbf := (Q_1,\ldots,Q_r)$ on $Y$ is a {\bf prolongation} of $\Pbf$ to $L$ (or to $Y$) provided that, for all $i = 1,\ldots,r$, the map $Y \rightarrow X$ sends $Q_i$ to $P_i$.
If $\Qbf$ is a prolongation of $\Pbf$, then the corresponding Parshin-chain $\vbf_\Qbf$ is a prolongation of $\vbf_\Pbf$ to $L$ as defined in \S\ref{subsec:parshin-chains}.

The existence of prolongations of $r$-prime-divisors is a slightly delicate matter, which we describe in the next lemma.
\begin{lemma}
\label{lemma:prolong-r-prime-divs}
Let $X$ be a normal $k$-variety with function field $K$.
Let $L$ be a finite extension of $K$ and let $Y$ denote the normalization of $X$ in $L$.
Let $\Pbf$ be an $r$-prime-divisor on $X$.
Then the following hold:
\begin{enumerate}
	\item Suppose $r = 1$ and write $\Pbf = (P)$.
	Then for any $Q \in Y_P$, the point $Q$ is a prime-divisor on $Y$ which is a prolongation of $P$.
	\item Suppose $r > 1$ and write $\Pbf = (P_1,\ldots,P_r)$.
	Let $\Qbf' = (Q_1,\ldots,Q_{r-1})$ be an $(r-1)$-prime-divisor on $Y$ which prolongs $\Pbf' = (P_1,\ldots,P_{r-1})$, and consider the finite cover $\overline{Q_{r-1}} \rightarrow \overline{P_{r-1}}$ induced by $Y \rightarrow X$.
	Then for all but finitely many prime-divisors $P_r$ of $\overline{P_{r-1}}$, any choice of $Q_r \in \left(\overline{Q_{r-1}}\right)_{P_r}$ yields an $r$-prime-divisor $\Qbf = (Q_1,\ldots,Q_r)$ which is a prolongation of $\Pbf$.
\end{enumerate}
\end{lemma}
\begin{proof}
\noindent\emph{Proof of (1):}
Since $Y \rightarrow X$ is a finite cover of normal $k$-varieties, the fiber $Y_P$ is finite and consists solely of points of codimension one in $Y$.
Each such point in $Y_P$ is regular since $Y$ is normal.
Thus, every point $Q \in Y_P$ is a prime divisor on $Y$ and $\Qbf = (Q)$ is a prolongation of $\Pbf$ by definition.

\vskip 5pt
\noindent\emph{Proof of (2):}
The induced map $\overline{Q_{r-1}} \rightarrow \overline{P_{r-1}}$ is a finite cover of (possibly non-normal) $k$-varieties.
Thus, there exists a non-empty open subset $U$ of $\overline{P_{r-1}}$ such that $U$ is contained in the regular locus of $\overline{P_{r-1}}$, and the preimage of $U$ in $\overline{Q_{r-1}}$, say $U'$, is contained in the regular locus of $\overline{Q_{r-1}}$.
Note that $U$ contains all but finitely many of the codimension one points of $\overline{P_{r-1}}$.

Let $P_r \in U$ be any such codimension one point; since $P_r$ is regular, we note that $\Pbf = (P_1,\ldots,P_r)$ is an $r$-prime-divisor on $X$.
Moreover, for all $Q_r \in \left(\overline{Q_{r-1}}\right)_{P_r} \subset U'$, the point $Q_r$ is regular and codimension one in $\overline{Q_{r-1}}$.
Thus $\Qbf = (Q_1,\ldots,Q_r)$ is an $r$-prime-divisor on $Y$ which is a prolongation of $\Pbf$.
\end{proof}

\subsection{$\ell$-unramified prime-divisors}
\label{subsec:ell-unramified-prime-divisors}

Suppose that $(K,v) \hookrightarrow (L,w)$ is an extension of \emph{discretely} valued fields.
We say that $w|v$ is {\bf $\ell$-unramified} provided that the ramification index $e(w|v)$ is not divisible by $\ell$.
Similarly, if $(K,v)$ is a discretely valued field and $L$ is a finite extension of $K$, we say that $v$ is {\bf $\ell$-unramified} in $L$ if $w|v$ is $\ell$-unramified for all prolongations $w$ of $v$ to $L$.

We make a similar definition for discrete Parshin-chains as follows.
Suppose that $L|K$ is a finite extension and $\vbf$ is a \emph{discrete} Parshin-chain of $K$.
Suppose that $\wbf$ is a prolongation of $\vbf$ to $L$ and recall that $\wbf$ is necessarily discrete.
We say that $\wbf|\vbf$ is {\bf $\ell$-unramified} provided that, setting $e(\wbf|\vbf) = (e_1,\ldots,e_r)$, the ramification index $e_i$ is not divisible by $\ell$ for all $i = 1,\ldots, r$. 
We say that $\vbf$ is {\bf $\ell$-unramified} in $L$ if $\wbf|\vbf$ is $\ell$-unramified for all prolongations $\wbf$ of $\vbf$ to $L$.

Suppose now that $K|k$ is a function field and that $X$ a model for $K$.
Let $L$ be a finite extension of $K$.
We say that a prime-divisor $P$ resp. $r$-prime-divisor $\Pbf$ on $X$ is {\bf $\ell$-unramified} in $L$ provided that the associated valuation $v_P$ resp. the associated Parshin-chain $\vbf_\Pbf$ is $\ell$-unramified in $L$.
The following Lemma shows that there are many $\ell$-unramified $r$-prime-divisors on a $k$-variety.
\begin{lemma}
\label{lemma:mod-ell-ramification}
Let $k$ be an algebraically closed field of characteristic $\neq \ell$.
Let $X$ be a $k$-variety with function field $K$ and let $L$ be a finite extension of $K$.
Then the following hold:
\begin{enumerate}
	\item All but finitely many prime-divisors $P$ on $X$ are $\ell$-unramified in $L$.
	\item Suppose that $\Pbf' = (P_1,\ldots,P_{r-1})$ is an $(r-1)$-prime-divisor on $X$ which is $\ell$-unramified in $L$.
	Then for all but finitely many prime-divisors $P_r$ of $\overline{P_{r-1}}$, the $r$-prime-divisor $\Pbf = (P_1,\ldots,P_r)$ is also $\ell$-unramified in $L$.
\end{enumerate}
\end{lemma}
\begin{proof}
\noindent\emph{Proof of (1):}
Let $L^s$ be the maximal separable subextension of $L|K$ so that $L^s|K$ is separable and $L|L^s$ is purely inseparable.
Since $L^s|K$ is separable, all but finitely many prime-divisors of $X$ are \emph{unramified} in $L^s$.

On the other hand, since $L|L^s$ is purely inseparable, for all extensions of discrete valuations $w|v$ of $L|L^s$, one has $e(w|v) = p^i$ where $p = \Char k \neq \ell$ and $i$ is some non-negative integer.
By the multiplicativity of ramification indices in towers of fields, we see that a prime-divisor $P$ of $X$ is $\ell$-unramified in $L^s$ if and only if $P$ is $\ell$-unramified in $L$.
Assertion (1) of the lemma follows. 

\vskip 5pt
\noindent\emph{Proof of (2):}
If $\wbf' = (w_1,\ldots,w_{r-1})$ is a prolongation of $\vbf_{\Pbf'}$ to $L$, then assertion (1) implies that all but finitely many prime-divisors $P_r$ on $\overline{P_{r-1}}$ are $\ell$-unramified in $\kappa(w_{r-1}) = \kappa(\wbf')$.
Since there are only finitely many prolongations $\wbf'$ of $\vbf_{\Pbf'}$ to $L$, assertion (2) follows.
\end{proof}

\subsection{Vanishing/non-vanishing in Milnor K-rings of function fields}
\label{subsec:non-vanishing}

\begin{lemma}
\label{lemma:milnor-dimension-of-function-fields}
Let $K$ be a function field over an algebraically closed field $k$ such that $\Char k \neq \ell$ and let $d := \trdeg(K|k)$.
Then for all $s > d$, one has $\k_{s}(K) = 0$.
\end{lemma}
\begin{proof}
%It suffices to prove that $\k_{d+1}(K) = 0$.
It is well-known that the $\ell$-cohomological dimension of $K$ is $d$ (cf. \cite{zbMATH01673451} \S4.2 Proposition 11).
The lemma now follows from the Bloch-Kato conjecture, which is now a theorem of {\sc Veovodsky-Rost} et al. \cite{Voevodsky2011}, \cite{Rost-chain-lemma}, \cite{zbMATH05594008}.
\end{proof}

\begin{remark}
\label{remark:bloch-kato}
It seems that the full force of the Voevodsky-Rost Theorem is not strictly required to prove Lemma \ref{lemma:milnor-dimension-of-function-fields}, as it suffices to prove that $\k_{r+1}(K) = 0$.
%For example, it is possible that one could follow an argument similar to the ``prime-to-characteristic'' argument in \cite{2005math.....11211A} Theorem 1.1, then use \cite{zbMATH03724572} \S3.3 Lemma 10.
However, the author is not aware of a proof of Lemma \ref{lemma:milnor-dimension-of-function-fields} which doesn't rely on the Voevodsky-Rost Theorem.
\end{remark}

\begin{lemma}
\label{lemma:independent-coords}
Let $K$ be a function field over an algebraically closed field $k$ such that $\Char k\neq\ell$.
Let $t_1,\ldots,t_r$ be given elements of $K$ which are algebraically independent over $k$.
Then there exist $a_1,\ldots,a_r \in k$ such that $\{t_1-a_1,\ldots,t_r-a_r\} \neq 0$ in $\k_r(K)$.

In particular, if $F_1,\ldots,F_r$ are subextensions of $K|k$ such that $\trdeg(F_i|k) = 1$ for all $i = 1,\ldots,r$ and, denoting the compositum $F_1 \cdots F_r$ by $F$, one has $\trdeg(F|k) = r$, then there exist $x_i \in F_i$ for $i = 1,\ldots,r$ such that $\{x_1,\ldots,x_r\} \neq 0$ in $\k_r(K)$.
\end{lemma}
\begin{proof}
By extending $t_1,\ldots,t_r$ to a transcendence base $\tbf = (t_1,\ldots,t_d)$ for $K|k$, we may assume without loss of generality that $r = d := \trdeg(K|k)$.
Consider $X := \A^d_{k,\tbf} = \Spec k[t_1,\ldots,t_d]$, affine $d$-space over $k$ with parameters $\tbf$.
Furthermore, consider the normalization $Y \rightarrow X$ of $X$ in the field extension $k(\tbf) \hookrightarrow K$.
In particular, $Y$ is a $k$-variety which is a normal model for $K|k$, and the map $Y \rightarrow X$ is finite and surjective.

For a closed point $x$ of $X = \A^d_{k,\tbf}$, say corresponding to the $k$-rational point $(a_1,\ldots,a_d) \in \A^d_{k,\tbf}(k)$, consider the system of regular parameters $\xbf(x) = (t_1-a_1,\ldots,t_d-a_d)$ at $x$, as well as the corresponding $d$-prime divisor $\Pbf_{x} := \Pbf_{\xbf(x)}$ on $X$ (cf. \S\ref{sub:regular-parameters}).
By Lemma \ref{lemma:mod-ell-ramification}, there exists a closed point $x_0$ in $X$ such that $\Pbf_{x_0}$ is $\ell$-unramified in the finite extension $K|k(\tbf)$.
For the rest of the proof, we will denote $\Pbf_{x_0}$ by $\Pbf$ and $\xbf(x_0)$ by $\xbf$ for such an $x_0$.

Let $\wbf$ be a prolongation of $\vbf := \vbf_\Pbf$ to $K$ and consider $e(\wbf|\vbf) = (e_1,\ldots,e_d)$.
Recall that $\xbf = (t_1-a_1,\ldots,t_d-a_d)$ is a uniformizing system for $\vbf$, and that, $\vbf$ being $\ell$-unramified in $K$, the ramification index $e_i$ is not divisible by $\ell$ for all $i = 1,\ldots,d$.
Moreover, by Lemma \ref{lemma:parshin-tame-symbols}, we see that $\{t_1-a_1,\ldots,t_d-a_d\}_\wbf = e_1 \cdots e_d$ in $\Z = \K_0(K\wbf)$.
Since $\ell$ doesn't divide $e_1 \cdots e_d$, we deduce that $\{t_1-a_1,\ldots,t_d-a_d\}_\wbf \neq 0$ in $\Z/\ell = \k_0(K\wbf)$.
In particular, $\{t_1-a_1,\ldots,t_d-a_d\} \neq 0$ in $\k_d(K)$, and this completes the proof of the lemma.
\end{proof}

The following proposition is needed to relate transcendence degree to the mod-$\ell$ Milnor dimension of relatively algebraically closed subextensions of $K|k$.
This will play a key role in \S\ref{sec:recovering-the-mod-ell-algebraic-lattice}.
For a related result/argument, see {\sc Pop} \cite{Pop2012a} Proposition 40(3).

\begin{proposition}
\label{prop:a1-ar-z-non-zero}
Let $K$ be a function field over an algebraically closed field $k$ such that $\Char k \neq \ell$, and let $d := \trdeg(K|k)$.
Let $L$ be a relatively algebraically closed subextension of $K|k$ such that $r := \trdeg(L|k) < d$.
Let $z \in K^\times \smallsetminus L^\times \cdot K^{\times \ell}$ be given.
Then there exist $x_1,\ldots,x_r \in L^\times$ such that $\{x_1,\ldots,x_r,z\} \neq 0$ in $\k_{r+1}(K)$.
\end{proposition}
\begin{proof}
Let $A$ be a regular $k$-algebra of finite type over $k$, whose fraction field is $L$; thus $S := \Spec A$ is an affine regular $k$-variety which is a model for $L|k$.
Choose a transcendence base $\tbf = (t_1,\ldots,t_s)$ for $K|L$, and let $K_0 = L(\tbf)$.
Thus, $K|K_0$ is a finite extension, although it is not necessarily separable.
Let $K_1$ denote the maximal separable subextension of $K|K_0$ so that $K|K_1$ is purely inseparable.
Since $\k_1(K_1) \rightarrow \k_1(K)$ is an isomorphism (Fact \ref{fact:inseperable-milnor-isomorphism}), we may assume without loss of generality that $z \in K_1$.
We let $K_2 = K_1[\sqrt[\ell]{z}]$ and note that $K_2|K_1$ is $\Z/\ell$-Galois and that $K_2|K_0$ is separable.
Since $L$ is algebraically closed in $K_1$, the assumption that $z \notin L^\times \cdot K^{\times \ell}$ ensures that $L$ is also algebraically closed in $K_2$; hence $K_2|L$ is a regular extension.

Let $B_0 = A[\tbf]$ be the polynomial algebra over $A$ in the variables $\tbf$.
We denote by $B_i$ the normalization of $B_0$ in $K_i$ for $i = 1,2$, and we denote $\Spec B_i$ by $X_i$ for $i = 0,1,2$.
Thus we have canonical maps $X_2 \rightarrow X_1 \rightarrow X_0 \rightarrow S = \Spec A$, with $k(X_i) = K_i$, and $X_i$ normal in its function field.
For each $i = 0,1,2$, the extension $K_i|L$ is regular, and therefore the fibers of $X_i \rightarrow S$ are generically geometrically integral.
Thus, we may replace $S$ with an affine open (i.e. replacing $A$ by a localization of the form $A[1/f]$ for some non-zero $f \in A$) and assume without loss of generality that $(X_i)_s$, the fiber of $X_i \rightarrow S$ over $s \in S$, is integral for all $s \in S$ and for each $i = 0,1,2$.
In particular, $(X_1)_s$ is irreducible and we will let $\eta_s \in X_1$ denote the generic point of $(X_1)_s$ for any $s \in S$.

Recall that $K_2 = K_1[\sqrt[\ell]{z}]$, and that $B_2$ is the normalization of $B_1$ in $K_2$.
Thus, there exists some non-zero $f_0 \in B_1$ such that $z \in B_1[1/f_0]$ and $B_2[1/f_0] = B_1[1/f_0,\sqrt[\ell]{z}]$.
By replacing $S$ with an affine open, we may further assume without loss of generality that for all points $s$ of $S$, the fiber $(X_1)_s$ of $X_1 \rightarrow S$ is not contained in $V(f_0)$, the zero-locus of this $f_0$.
This implies that for all $s \in S$, one has $z \in \Oc_{X_1,\eta_s}$, and $k((X_2)_s) = k((X_1)_s)[\sqrt[\ell]{\bar z}]$, where $\bar z$ denotes the image of $z$ under the canonical map $\Oc_{X_1,\eta_s} \rightarrow k(\eta_s) = k((X_1)_s)$.

Since $K_2|L$ is a regular extension, by the Bertini-Noether theorem (cf. \cite{fried2006field} Proposition 9.4.3), we can replace $S$ by an affine open and assume without loss of generality that for all closed points $s$ of $S$, one has $[k((X_2)_s) : k((X_0)_s)] = [K_2:K_0]$.
In particular, this implies that for all closed points $s$ of $S$, one has $[k((X_2)_s):k((X_1)_s)] = [K_2:K_1] = \ell$.
By the discussion above, we deduce that for all closed points $s$ of $S$, the image $\bar z$ of $z$ in $k((X_1)_s))^\times$ is not an $\ell$-th power in $k((X_1)_s)^\times$.
In particular, note that this implies $z \in \Oc_{X_1,\eta_s}^\times$ for all closed points $s \in S$.

Now let $s$ be a closed point of $S$ and let $x_1,\ldots,x_r$ be a regular system of parameters for $s$.
Then $\xbf = (x_1,\ldots,x_r)$ is a \emph{partial system of regular parameters} of some closed point on $X_0$ which lies in $(X_0)_s$.
Thus, we may produce the associated $r$-prime divisor $\Pbf_\xbf$ on $X_0$ (cf. \S\ref{sub:regular-parameters}), as well as the associated discrete Parshin-chain $\vbf := \vbf_{\Pbf_\xbf}$ on $K_0 = L(\tbf)$.
We see from the construction that $\kappa(\vbf_\xbf)$ is precisely $k((X_0)_s)$.
Using Lemmas \ref{lemma:prolong-r-prime-divs} and \ref{lemma:mod-ell-ramification}, it is clear that we can choose the regular closed point $s$ of $S$ and the regular system of parameters $\xbf$ for $s$ in such a way so that the following two conditions hold:
\begin{enumerate}
	\item The Parshin-chain $\vbf_{\Pbf_\xbf} = \vbf$ is $\ell$-unramified in $K_1$.
	\item The $r$-prime divisor $\Pbf_\xbf$ on $X_0$ has a prolongation $\Qbf = (Q_1,\ldots,Q_r)$ which is an $r$-prime divisor on $X_1$ such that $Q_r = \eta_s$ is the generic point of $(X_1)_s$. 
\end{enumerate}
Setting $\wbf := v_\Qbf$, the discrete Parshin-chain associated to $\Qbf$ as in (2) above, we see that $\wbf$ is a prolongation of $\vbf$ to $K_1$, and that $\kappa(\wbf) = k((X_1)_s)$.
Furthermore, we have $z \in U_{\wbf}$ since $z \in \Oc_{X_1,\eta_s}^\times$.

Now let $\xbf = (x_1,\ldots,x_r)$ and $\wbf|\vbf$ be as above.
Setting $e(\wbf|\vbf) = (e_1,\ldots,e_r)$, we recall that for all $i = 1,\ldots,r$, the ramification index $e_i$ is not divisible by $\ell$, since $\vbf$ is $\ell$-unramified in $K_1$.
By Lemma \ref{lemma:parshin-tame-symbols}, we see that
\[ \{x_1,\ldots,x_r,z\}_\wbf = e_1 \cdots e_r \cdot \bar z \in \K_1(\kappa(\wbf)) = \K_1(k((X_1)_s)).\]
But $e_1 \cdots e_r$ is not divisible by $\ell$, and recall that $\bar z$ is not an $\ell$-th power in $\kappa(\wbf)^\times = k((X_1)_s)^\times$.
In particular, $\{x_1,\ldots,x_r,z\}_\wbf$ is non-trivial as an element of $\k_1(\kappa(\wbf))$.
We deduce that $\{x_1,\ldots,x_r,z\} \neq 0$ in $\k_{r+1}(K_1)$.
Finally, since $K|K_1$ is purely inseparable, Fact \ref{fact:inseperable-milnor-isomorphism} implies that $\{x_1,\ldots,x_r,z\} \neq 0$ in $\k_{r+1}(K)$.
This completes the proof of the proposition.
\end{proof}

\section{Graded Lattices}
\label{sec:graded_lattices}

A {\bf graded lattice} $\Lf^* = \coprod_{r \geq 0}\Lf^r$ is a partially ordered graded set which satisfies the following properties: 
\begin{enumerate}
	\item Every subset $S \subset \Lf^*$ has a least upper bound $\vee S$ and a greatest lower bound $\wedge S$. 
	Namely, $\Lf^*$ is a \emph{complete} lattice.
	\item If $a \in \Lf^r$ and $b \in \Lf^s$ are such that $a < b$, then $r < s$.
	Namely, the partial ordering of $\Lf^*$ respects the grading.
\end{enumerate}
An isomorphism of graded lattices $f^* : \Lf^*_1 \rightarrow \Lf^*_2$ is a graded bijection $f^* = \coprod_{i \geq 0} f^i$ on the underlying sets which respects the partial ordering.
The set of isomorphisms $\Lf^*_1 \rightarrow \Lf^*_2$ will be denoted by $\Isom^*(\Lf^*_1,\Lf^*_2)$.

\subsection{The graded lattice of algebraic closure}
\label{sub:algebraic_closure}

Let $K|k$ be an arbitrary extension of fields of finite transcendence degree and assume that $k$ is algebraically closed in $K$.
We define a graded lattice $\G^*(K|k) = \coprod_{r \geq 0} \G^r(K|k)$, as follows:
\begin{enumerate}
	\item $\G^r(K|k) := \{ \overline{k(t_1,\ldots,t_r)} \cap K \ : \ t_1,\ldots,t_r \in K, \ \trdeg(k(t_1,\ldots,t_r)|k) = r\}$.
	\item The partial ordering on $\G^*(K|k)$ is induced by inclusion of subextensions of $K|k$.
\end{enumerate}
Namely, $\G^*(K|k)$ is the partially ordered set of all relatively algebraically closed subextensions of $K|k$, while the graded component $\G^r(K|k)$ consists of the relatively algebraically closed subextensions of $K|k$ which have transcendence degree $r$ over $k$.

For a subset $S \subset\G^*(K|k)$, define $k(S)$ to be the compositum of the elements of $S$ as subfields of $K$, and let $\kappa_S := \overline{k(S)} \cap K$ be the algebraic closure of $k(S)$ in $K$.
We immediately see that $\kappa_S$ is an element of $\G^*(K|k)$, and that $\kappa_S = \vee S$ is the least upper bound of $S$.
The greatest lower bound $\wedge S$ of $S$ is simply $\cap S$, the intersection of the elements of $S$ considered as subfields of $K$ which contain $k$.

\subsection{The map $\Omega_K$ and algebraic closure mod-$\ell$}
\label{sub:algebraic_closure_mod_}

Let $K|k$ be an arbitrary extension of fields of finite transcendence degree and assume that $k$ is algebraically closed in $K$.
For a subset $S$ of $K$, we define $\Omega_K(S)$ to be the \emph{image} of the composition: 
\[ S \cap K^\times \hookrightarrow K^\times \twoheadrightarrow K^\times/\ell = \k_1(K). \]
Thus, $\Omega_K(S)$ is always a subset of $\k_1(K)$, and $\Omega_K$ is \emph{weakly monotone}: $S \subsetneq T \subset K$ implies $\Omega_K(S) \subset \Omega_K(T)$.

For a given $r \geq 0$, define
\[ \Gf^r(K|k) := \{ \Omega_K(L) \ : \ L \in \G^r(K|k) \}. \]
Thus, each $\Gf^r(K|k)$ is a collection of subgroups of $\k_1(K)$.
We furthermore define 
\[\Gf^*(K|k) := \bigcup_{r \geq 0} \Gf^r(K|k)\]
with the union being taken inside the power-set of $\k_1(K)$.
Namely, $\Gf^*(K|k)$ is also a collection of subgroups of $\k_1(K)$, \emph{without repetition}.
We consider $\Gf^*(K|k)$ as a partially ordered set where the ordering is given by inclusion of subgroups of $\k_1(K)$.

However, note that $\Gf^*(K|k)$ need not be a graded set whose graded components are $\Gf^r(K|k)$.
Namely, if there exist $L_1 \in \G^r(K|k)$ and $L_2 \in \G^s(K|k)$ with $r \neq s$ such that $\Omega_K(L_1) = \Omega_K(L_2)$, then $\Gf^r(K|k)$ and $\Gf^s(K|k)$ will have a non-trivial intersection.
Nevertheless, $\Omega_K$ is a surjective map $\G^*(K|k) \rightarrow \Gf^*(K|k)$ which (weakly) preserves the ordering: if $L_1,L_2 \in \G^*(K|k)$ with $L_1 < L_2$ in the ordering of $\G^*(K|k)$, then $\Omega_K(L_1) \leq \Omega_K(L_2)$ in the ordering of $\Gf^*(K|k)$.

\subsection{Function fields}
\label{sub:function-fields-lattices}

The following proposition shows that $\Gf^*(K|k)$ is actually a graded lattice which is isomorphic to $\G^*(K|k)$ in the special case where $K$ is a function field over an algebraically closed field $k$ of characteristic $\neq \ell$.

\begin{proposition}
\label{prop:alg_closure_mod_ell_lattice}
	Let $K$ be a function field over an algebraically closed field $k$ such that $\Char k \neq \ell$.
	Then the following hold:
	\begin{enumerate}
		\item Suppose that $L_1,L_2 \in \G^*(K|k)$ are given.
		Then $L_1 = L_2$ if and only if $\Omega_K(L_1) = \Omega_K(L_2)$.
		In particular, if $r \neq s$ then $\Gf^r(K|k)$ and $\Gf^s(K|k)$ are disjoint subsets of $\Gf^*(K|k)$, and thus $\Gf^*(K|k) = \coprod_{r \geq 0} \Gf^r(K|k)$ is a graded partially-ordered set.
		\item The graded partially ordered set $\Gf^*(K|k)$ is a graded lattice, and the map $\Omega_K : \G^*(K|k) \rightarrow \Gf^*(K|k)$ is an isomorphism of graded lattices.
	\end{enumerate}
\end{proposition}
\begin{proof}
\noindent{\emph{Proof of (1):}}
Say $r$ and $s$ are non-negative integers such that $L_1 \in \G^r(K|k)$ and $L_2 \in \G^s(K|k)$, and assume that $\Omega_K(L_1) = \Omega_K(L_2)$.
By Lemmas \ref{lemma:milnor-dimension-of-function-fields} and \ref{lemma:independent-coords}, we see that $\dimm_K(\Omega_K(L_1)) = r$ and $\dimm_K(\Omega_K(L_2)) = s$, where $\dimm_K$ is as defined in \S\ref{sub:milnor_dimension}.
Since $\Omega_K(L_1) = \Omega_K(L_2)$, we deduce that $r = s$.

Now assume for a contradiction that $L_1 \neq L_2$.
Since $r = s$, this implies that the fields $L_1$ and $L_2$ are incomparable as subfields of $K$ because $L_1$ and $L_2$ are algebraically closed in $K$.
In particular, $\trdeg(L_1 L_2 | k) > r$.
Thus, by Lemma \ref{lemma:independent-coords}, there exist $x_1,\ldots,x_r \in L_1$ and $y \in L_2$ such that $\{x_1,\ldots,x_r,y\} \neq 0$ in $\k_{r+1}(K)$.
Since $\Omega_K(L_1) = \Omega_K(L_2)$, this implies that $\dimm_K(\Omega_K(L_1)) > r$, which contradicts the fact that $\dimm_K(\Omega_K(L_1)) = r$.

\vskip 5pt
\noindent{\emph{Proof of (2):}}
By assertion (1), we know that $\Gf^*(K|k)$ is a graded partially ordered set whose graded components are $\Gf^r(K|k)$.
It's clear from the definition that $\Omega_K : \G^*(K|k) \rightarrow \Gf^*(K|k)$ is a surjective map which respects the grading, and by assertion (1) this map is also injective.
Finally, it's clear from the definition that $\Omega_K$ respects the partial ordering.
Since $\G^*(K|k)$ is a graded lattice, and $\Omega_K : \G^*(K|k) \rightarrow \Gf^*(K|k)$ is a bijection which respects the grading and partial ordering, we deduce that $\Gf^*(K|k)$ is also a graded lattice, and that $\Omega_K$ is actually an isomorphism of graded lattices.
\end{proof}

\section{Recovering the Mod-$\ell$ Lattice}
\label{sec:recovering-the-mod-ell-algebraic-lattice}

Let $K$ be a function field over an algebraically closed field $k$ such that $\Char k \neq \ell$ and recall that $\Gf^*(K|k)$ is a graded lattice by Proposition \ref{prop:alg_closure_mod_ell_lattice}.
In this section we will show how to reconstruct this graded lattice using the mod-$\ell$ Milnor K-theory of $K$ along with the collection of rational subgroups of $\k_1(K)$.
We begin by recalling the definition of \emph{general elements} (terminology due to {\sc Pop} \cite{Pop2012a}, \cite{Pop2011}) and \emph{rational subgroups}.

An element $t \in K \smallsetminus k$ is called a {\bf general element} of $K$ provided that $k(t)$ is algebraically closed in $K$.
Note that this condition is not intrinsic to the element $t$, but rather it completely depends on the ambient field $K$.
Nevertheless, if $L$ is a subextension of $K|k$ such that $L$ is algebraically closed in $K$ and $t \in L \smallsetminus k$, then $t$ is general in $L$ if and only if $t$ is general in $K$.
Also, note that if $t$ is a general element of $K$, then $t^{-1}$, $t+a$ and $a \cdot t$ are also general elements of $K$ for all $a \in k^\times$.

If $M|F$ is an extension of arbitrary fields, we say that an element $x \in M$ is {\bf separable in $M$ over $F$} if $x \notin M^p F$ where $p = \Char F$.
By convention, if $\Char F = 0$, then every $x \in M$ is separable in $M$ over $F$.
If $F$ is understood from context, we may also omit the ``over $F$'' from the terminology.
Note in particular that if $F$ is perfect and $\Char F = p > 0$, then $x \in M$ is separable in $M$ (over $F$) if and only if $x \notin M^p$.
It is important to note that this terminology differs from the usual notion of a \emph{separable element} in an algebraic extension.
Nevertheless, this terminology was introduced by {\sc Pop} in \cite{Pop2012a}, so we use it here for the sake of consistency.

In the case where $K$ is a function field over an algebraically closed field $k$, we see that any general element of $K$ is also separable in $K$.
In fact, any non-constant element of $K$ is a power of some element of $K$ which is separable in $K$.
Moreover, if $L$ is a relatively algebraically closed subextension of $K|k$ and $t \in L$, then $t$ is separable in $L$ if and only if $t$ is separable in $K$.
The existence of many general elements is guaranteed by the following so-called \emph{Birational Bertini Theorem}, which works in the more general situation of regular function fields over arbitrary infinite fields.

\begin{fact}[Birational Bertini Theorem -- cf. \cite{lang1972introduction} Ch. VIII, pg. 213]
\label{fact:birational-bertini}
Let $F$ be an arbitrary infinite field and let $M$ be a regular finitely-generated field extension of $F$.
Let $x,y \in M$ be algebraically independent over $F$ with $x$ separable in $M$.
Then for all but finitely many $a \in F$, the field $M$ is a regular extension of $F(ax+y)$, and in particular $F(ax+y)$ is algebraically closed in $M$.
\end{fact}

Going back to the situation where $K$ is a function field over an algebraically closed field $k$ such that $\Char k \neq \ell$, we say that $A$ is a {\bf rational subgroup} of $\k_1(K)$ provided that there exists a general element $t$ of $K$ such that $A = \Omega_K(k(t))$.
The collection of rational subgroups of $\k_1(K)$ is denoted by $\Gfor(K|k)$, and we note that $\Gfor(K|k)$ is a subset of $\Gf^1(K|k)$, in the notation of \S\ref{sub:algebraic_closure_mod_}.

\subsection{Recovering higher-dimensional subsets}
\label{subsection:recovering-higher-dimensionals-subsets}

In this subsection we show how to recover $\Gf^r(K|k)$ for $r \geq 2$ using $\k_*(K)$ and $\Gfor(K|k)$.
We first need a lemma which easily follows from Fact \ref{fact:birational-bertini}.

\begin{lemma}
\label{lemma:birational-bertini-reformulation}
Let $K$ be a function field over an algebraically closed field $k$, and let $L$ be a subextension of $K|k$ such that $\trdeg(L|k) \geq 2$ and $L$ is algebraically closed in $K$.
Then any element of $L^\times$ is a product of two elements of $L$ which are general in $K$.
\end{lemma}
\begin{proof}
First suppose that $a \in k^\times$.
Since $\trdeg(L|k) \geq 2$, by Fact \ref{fact:birational-bertini}, there exists an element $t \in L$ such that $t$ is general in $K$.
Observe that $t^{-1} \cdot a$ is also general in $K$ and thus $a = t \cdot (t^{-1} \cdot a)$ is a product of two elements of $L$ which are general in $K$.

Now let $x \in L^\times \smallsetminus k^\times$ be given.
Since $\trdeg(L|k) \geq 2$, we may choose $y \in L$ which is algebraically independent from $x$ over $k$.
Moreover, since $L$ is relatively algebraically closed in $K$ and since any non-constant element of $K$ is a power of an element which is separable in $K$, we may assume without loss of generality that $y$ is separable in $K$.
By Fact \ref{fact:birational-bertini}, we may replace $y$ by $x+ay$, for some $a \in k^\times$, and assume furthermore that $y$ is general in $K$.
Namely, $y$ is general (hence separable) in $K$, and $x,y$ are algebraically independent over $k$.
Note that $1/x$ or $y/x$ must be separable since $y$ is separable.
Thus, using Fact \ref{fact:birational-bertini} again, we see that there exists some $a \in k$ such that $(y/x) + a \cdot (1/x) = (y+a)/x$ is also general in $K$.
On the other hand, since $y$ is general in $K$, the element $y+a$ is also general in $K$.
Moreover, since $(y+a)/x$ is general in $K$, its inverse $x/(y+a)$ is also general in $K$.
Thus, $x = (x/(y+a)) \cdot (y+a)$ is a product of two elements of $L$ which are general in $K$.
\end{proof}

\begin{proposition}
\label{prop:recovering-higher-dimensions}
Let $K$ be a function field over an algebraically closed field $k$ such that $\Char k \neq \ell$ and assume that $\trdeg(K|k) \geq 2$.
Let $A$ be a subgroup of $\k_1(K)$ and let $r \geq 2$ be given.
Then the following are equivalent:
\begin{enumerate}
	\item One has $A \in \Gf^r(K|k)$.
	\item The subgroup $A$ is maximal among subgroups $B$ of $\k_1(K)$ which satisfy the following two conditions:
	\begin{enumerate}
		\item There exist a subset $S \subset \Gfor(K|k)$ such that $B = \langle T  \ : \ T \in S \rangle$.
		\item One has $\dimm_K(B) = r$.
	\end{enumerate}
\end{enumerate}
\end{proposition}
\begin{proof}
\vskip 5pt
\noindent{$(1) \Rightarrow (2)$:}
Say $L \in \G^r(K|k)$ is given such that $A = \Omega_K(L)$.
Setting $S_A = \{ T \in \Gfor(K|k) \ : \ T \subset A\}$, it follows from Lemma \ref{lemma:birational-bertini-reformulation} that $A = \langle T  \ : \ T \in S_A \rangle$.
Thus condition (a) holds for $A$.
Also, condition (b) holds for $A$ by Lemmas \ref{lemma:milnor-dimension-of-function-fields} and \ref{lemma:independent-coords}.

We must now show that $A$ is maximal with respect to conditions (a) and (b).
Suppose that $B$ also satisfies (a) and (b) and that $A \subset B$.
Assume for a contradiction that $A \neq B$.
Thus, by condition (a), there exists some $T_0 \in \Gfor(K|k)$ such that $T_0 \not\subset A$ and $T_0 \subset B$.
Let $F_0 \in \G^1(K|k)$ be given such that $\Omega_K(F_0) = T_0$.
Since $T_0 \not\subset A$, we see that $F_0$ is not contained in $L$.
Furthermore, since $L$ is algebraically closed in $K$, we deduce that $\trdeg(F_0 L|k) = r+1$.
By Lemma \ref{lemma:independent-coords}, we see that there exist $x_1,\ldots,x_r \in L$ and $y \in F_0$ such that $\{x_1,\ldots,x_r,y\} \neq 0$ in $\k_{r+1}(K)$.
But this implies that $\dimm_K(B) > r$, which contradicts condition (b).
Having obtained our contradiction, we deduce that $A$ is indeed maximal among subgroups of $\k_1(K)$ which satisfy conditions (a), (b).

\vskip 5pt
\noindent{$(2) \Rightarrow (1)$:} 
Suppose that $A$ is maximal among subgroups satisfying conditions (a) and (b).
Let $S$ be a subset of $\Gfor(K|k)$ such that $A = \langle T \ : \ T \in S \rangle$.
For each $T \in S$, choose a $t_T \in K \smallsetminus k$ which is general in $K$ such that $T = \Omega_K(k(t_T))$.
Also, let $M$ be the subfield of $K$ generated by $k$ and all these $t_T$ as $T \in S$ varies.
Note that $A \subset \Omega_K(M)$ by condition (a).
By Lemmas \ref{lemma:milnor-dimension-of-function-fields} and \ref{lemma:independent-coords}, we see that condition (b) implies $\trdeg(M|k) = r$.
Thus, there exists some $L \in \G^r(K|k)$ such that $M \subset L$, and thus $A \subset \Omega_K(L)$.
On the other hand, $\Omega_K(L)$ satisfies conditions (a) and (b) by the argument above.
The maximality of $A$ implies that $A = \Omega_K(L)$ and thus $A \in \Gf^r(K|k)$.
\end{proof}

\subsection{Recovering one-dimensional subsets}
\label{subsection:recovering-points}

The next proposition shows how to recover $\Gf^1(K|k)$ using $\Gf^r(K|k)$ for $r = 2,3$, the ring $\k_*(K)$, and $\Gfor(K|k)$.
To state this proposition, we recall from \S\ref{sub:milnor_dimension} that $\Mf^1(K)$ denotes the collection of subgroups $A$ of $\k_1(K)$ which are maximal among subgroups $B$ of $\k_1(K)$ such that $\dimm_K(B) = 1$.

\begin{proposition}
\label{prop:recovering-points}
Let $K$ be a function field over an algebraically closed field $k$ such that $\Char k \neq \ell$, and assume that $\trdeg(K|k) \geq 4$.
Let $A$ be a subgroup of $\k_1(K)$.
Then the following are equivalent:
\begin{enumerate}
	\item One has $A \in \Gf^1(K|k)$.
	\item One has $A \in \Mf^1(K)$, and there exist some $D \in \Gfor(K|k)$, $B_1,B_2, E \in \Gf^2(K|k)$, and $B_1', B_2', C \in \Gf^3(K|k)$ such that the following conditions hold:
	\begin{enumerate}
		\item $B_1 \neq B_2$ and $D \not\subset C$.
		\item $B_1 \cup B_2 \subset C$.
		\item $B_1 \cup D \subset B'_1$ and $B_2 \cup D \subset B'_2$.
		\item $E \subset B_1' \cap B_2'$.
		\item $A = B_1 \cap B_2$.
	\end{enumerate}
\end{enumerate}
\end{proposition}
\begin{proof}
First we show that $\Gf^1(K|k) \subset \Mf^1(K)$.
Let $A \in \Gf^1(K|k)$ be given.
By Lemma \ref{lemma:milnor-dimension-of-function-fields} and the non-triviality of $A$, it follows that $\dimm_K(A) = 1$.
Suppose now that $z \in \k_1(K) \smallsetminus A$.
By Proposition \ref{prop:a1-ar-z-non-zero}, there exists an element $a \in A$ such that $\{a,z\} \neq 0$ in $\k_2(K)$.
In particular, $\dimm_K(\langle A,z\rangle) \geq 2$.
This shows that $A$ is indeed maximal among subgroups $B$ of $\k_1(K)$ such that $\dimm_K(B) = 1$, and thus $A \in \Mf^1(K)$.

\vskip 5pt
\noindent{$(1) \Rightarrow (2)$:}
The argument above shows that $A \in \Mf^1(K)$.
We must therefore prove the existence of $D,B_1,B_2,E,B_1',B_2',C$ which satisfy the other assertions of condition (2).
Let $F \in \G^1(K|k)$ be given such that $A = \Omega_K(F)$.
Let $\tbf = (t_1,\ldots,t_r)$ be a transcendence base for $K|F$, and note that $r \geq 3$ since $\trdeg(K|k) \geq 4$ by assumption.
Let $K'$ denote the maximal separable subextension of $K|F(\tbf)$.
By replacing $t_1$ resp. $t_2$ by elements of the form $t_1 + a t_3$ resp. $t_2 + b t_3$ for some $a,b \in k$, we may assume (by Fact \ref{fact:birational-bertini}) that $K'$ is a regular extension of $M := F(t_1,t_2)$.
Let $M_1 := F(t_1)$ and $M_2 := F(t_2)$ and note that both $M_1$ and $M_2$ are algebraically closed in $K'$.

\begin{claim*}
	One has $\Omega_{K'}(F) = \Omega_{K'}(M_1) \cap \Omega_{K'}(M_2)$.
\end{claim*}
\begin{proof}
	Since $M = F(t_1,t_2)$ is rational over $F$ and $M_i = F(t_i)$ for $i = 1,2$, it is easy to see that $\Omega_M(F) = \Omega_M(M_1) \cap \Omega_M(M_2)$, as follows.
	First, observe that we have a short exact sequence
	\[ 0 \rightarrow \k_1(F) \rightarrow \k_1(M) \xrightarrow{D} \bigoplus_{P} \Z/\ell \cdot [P] \rightarrow 0, \]
	where $P$ varies over the monic irreducible polynomials in $F[t_1,t_2]$, and the component map $\k_1(M) \rightarrow \Z/\ell \cdot [P]$ associated to $P$ is the reduction of the $P$-adic valuation modulo $\ell$-th powers.
	The images $D(\Omega_M(M_1))$ and $D(\Omega_M(M_2))$ clearly intersect trivially in the right-hand-side of this exact sequence.
	Since $\Omega_M(F)$ is contained in $\Omega_M(M_1) \cap \Omega_M(M_2)$, we deduce that $\Omega_M(F) = \Omega_M(M_1) \cap \Omega_M(M_2)$.
	To deduce the original claim, recall that the map $\k_1(M) \rightarrow \k_1(K')$ is injective (as $M$ is algebraically closed in $K'$).
	The claim follows.
\end{proof}

For $i = 1,2$, let $F_i$ denote the algebraic closure of $M_i$ in $K$ and let $L$ denote the algebraic closure of $M$ in $K$.
Since $K|K'$ is purely inseparable, we note that, for $i = 1,2$, the extension $F_i | M_i$ is also purely inseparable.

\begin{claim*}
	One has $\Omega_K(F_1) \cap \Omega_K(F_2) = \Omega_K(F)$.
\end{claim*}
\begin{proof}
Consider the following commutative diagram
\[
\xymatrix{
\k_1(F) \ar[r] \ar[d] & \k_1(M_1)  \ar[dr]^{\cong} \ar[d] \\
\k_1(M_2) \ar[dr]_\cong \ar[r] & \k_1(K') \ar[dr]^\cong	& \k_1(F_1) \ar[d] \\
{} & \k_1(F_2) \ar[r] & \k_1(K)
}
\]
and note that the diagonal arrows are isomorphisms by Fact \ref{fact:inseperable-milnor-isomorphism}.
By the previous claim, we see that the top-left square in this diagram is a Cartesian square.
The claim now follows easily since $\Omega_K(F)$ is the image of $\k_1(F) \rightarrow \k_1(K)$ and $\Omega_K(F_i)$ is the image of $\k_1(F_i) \rightarrow \k_1(K)$ for $i = 1,2$.
\end{proof}

We now show how to find the various subgroups of $\k_1(K)$ which satisfy properties (a)-(e).
We let $B_1 := \Omega_K(F_1)$, $B_2 := \Omega_K(F_2)$, and $C := \Omega_K(L)$.
Let $t$ be a \emph{general} element of $K$ which is transcendental over $L$; the existence of such an element is guaranteed by Fact \ref{fact:birational-bertini}.
Finally, let $B'_1 := \Omega_K(\overline{F_1(t)} \cap K)$, $B'_2 := \Omega_K(\overline{F_2(t)} \cap K)$, $E := \Omega_K(\overline{F(t)} \cap K)$ and $D := \Omega_K(\overline{k(t)} \cap K) = \Omega_K(k(t))$.
Property (a) follows from Proposition \ref{prop:alg_closure_mod_ell_lattice}, since $F_1 \neq F_2$ and since $t$ is transcendental over $L$.
Properties (b), (c) and (d) are trivial to check as they just follow from the corresponding inclusion of subfields of $K$.
Lastly, property (e) is the claim above.

\vskip 5pt
\noindent{$(2) \Rightarrow (1)$:}
Proposition \ref{prop:alg_closure_mod_ell_lattice} implies the following fact which we tacitly use for the rest of the proof: Given $W_1,W_2 \in \G^*(K|k)$, one has $W_1 \subset W_2$ if and only if $\Omega_K(W_1) \subset \Omega_K(W_2)$.

Choose $F_1,F_2 \in \G^2(K|k)$ such that $\Omega_K(F_i) = B_i$ and $L \in \G^3(K|k)$ such that $\Omega_K(L) = C$.
Property (b) implies that $F_1 F_2 \subset L$.
Choose $t$ a general element of $K$ such that $\Omega_K(k(t)) = \Omega_K(\overline{k(t)} \cap K) = D$.
Property (a) implies that $F_1 \neq F_2$ and that $t$ is transcendental over $L$.
Moreover, property (c) implies that $\Omega_K(\overline{F_1(t)} \cap K) = B'_1$ and $\Omega_K(\overline{F_2(t)} \cap K) = B'_2$.
Choose $M \in \G^2(K|k)$ such that $\Omega_K(M) = E$.
Then $M \subset \overline{F_1(t)} \cap \overline{F_2(t)} \cap K$ by property (d).

Let $F = F_1 \cap F_2$.
Since $F_1 \neq F_2$, we note that $\trdeg(F|k) \leq 1$.
Moreover, since $F_1 F_2 \subset L$ and $t$ is transcendental over $L$, we see that $\overline{F(t)} \cap K = \overline{F_1(t)} \cap \overline{F_2(t)} \cap K$.
But $t$ is transcendental over $F$, and thus 
\[ 1 \leq \trdeg(\overline{F(t)} \cap K|k) = \trdeg(F|k) + 1 \leq 2. \]
On the other hand, we have $M \subset \overline{F_1(t)} \cap \overline{F_2(t)} \cap K = \overline{F(t)} \cap K$.
Since $\trdeg(M|k) = 2$ and $M$ is algebraically closed in $K$, we see that $M = \overline{F(t)} \cap K$, and thus $\trdeg(F|k) = 1$.

To conclude the proof, first note that $\Omega_K(F) \subset A$ by property (e).
Also, the argument at the start of the proof (using Proposition \ref{prop:a1-ar-z-non-zero}) shows that $\Omega_K(F)$ is an element of $\Mf^1(K)$ since $F \in \G^1(K|k)$.
By the ``maximality'' in the definition of $\Mf^1(K)$, any two comparable elements of $\Mf^1(K)$ must be identical.
Since $A \in \Mf^1(K)$ by assumption and $\Omega_K(F) \subset A$, we deduce that $A = \Omega_K(F)$, hence $A$ is an element of $\Gf^1(K|k)$.
\end{proof}

\section{Combinatorial Geometries}
\label{section:combinatorial-geometries}

The last key step in the proof of the ``Milnor variant'' of the main theorem is to show how to construct the combinatorial geometry of (relative) algebraic closure, which was considered in \cite{Evans1991}, \cite{Evans1995} and \cite{Gismatullin2008}, using our K-theoretic data.
This will be accomplished via the standard construction which produces a set with a closure operation from a graded lattice.
We call this construction the {\bf $\Cf$-construction}; that is, if $\Lf^*$ is a graded lattice, we denote by $\Cf(\Lf^*)$ the associated set with closure operation.
When applied to $\G^*(K|k)$, this construction produces the combinatorial geometry considered in loc.cit.

\subsection{Closure operations}
\label{subsection:closure-operations}

Let $S$ be a set and let $\Pc(S)$ denote the power set of $S$.
A {\bf closure operation} on $S$ is a function $\cl : \Pc(S) \rightarrow \Pc(S)$ such that for all subsets $A \subset B \subset S$, one has $A \subset \cl(A) = \cl(\cl(A))$ and $\cl(A) \subset \cl(B)$.
If we write $(S,\cl)$, we implicitly mean that $S$ is a set which is equipped with a closure operation $\cl$.
Also, for a finite subset $\{a_1,\ldots,a_n\} \subset S$, we will usually write $\cl(a_1,\ldots,a_n)$ instead of $\cl(\{a_1,\ldots,a_n\})$.
Finally, we define an isomorphism of sets with closure operations $f : (S_1,\cl_1) \rightarrow (S_2,\cl_2)$ to be a bijection of the underlying sets $f : S_1 \rightarrow S_2$ which is compatible with closures in the sense that $f(\cl_1(A)) = \cl_2(f(A))$ for all $A \subset S_1$.
The set of isomorphisms $(S_1,\cl_1) \rightarrow (S_2,\cl_2)$ will be denoted by $\Isom^{\cl}(S_1,S_2)$.

Suppose that $\Lf^*$ is a graded lattice.
We can associate to $\Lf^*$ a closure operation $\cl$ on $\Lf^1$, as follows.
For a subset $A \subset \Lf^1$, define 
\[ \cl(A) = \{ a \in \Lf^1 \ : \ a \leq \vee A \}. \]
Since $\vee A$ is the least upper bound of $A$, we see that $A \subset \cl(A)$.
This implies that $\vee \cl(A) = \vee A$ and thus $\vee \cl(\cl(A)) = \vee(\cl(A))$.
Therefore, we deduce that $\cl(A) = \cl(\cl(A))$.
Furthermore, if $A \subset B$ then $\cl(A) \subset \cl(B)$ since $\vee A \leq \vee B$.
Thus $\cl$ is indeed a closure operation on $\Lf^1$.
We will denote $(\Lf^1,\cl)$ by $\Cf(\Lf^*)$.

Finally, note that the construction above is compatible with isomorphisms.
In particular, any isomorphism of graded lattices $f^* : \Lf_1^* \rightarrow \Lf_2^*$ induces an isomorphism of sets with closure operations $\Cf(f^*) : \Cf(\Lf_1^*) \rightarrow \Cf(\Lf_2^*)$ where $f^1 : \Lf_1^1 \rightarrow \Lf_2^1$ is the corresponding bijection on the underlying sets.
Namely, we get a canonical map of isomorphism sets:
\[ \Cf : \Isom^*(\Lf_1^*,\Lf_2^*) \rightarrow \Isom^{\cl}(\Cf(\Lf_1^*),\Cf(\Lf_2^*)). \]
Moreover, $\Cf$ is compatible with compositions of isomorphisms in the sense that, if $f_{12}^* : \Lf_1^* \rightarrow \Lf_2^*$ and $f_{23}^* : \Lf_2^* \rightarrow \Lf_3^*$ are isomorphisms of graded lattices, then $\Cf(f_{23}^* \circ f_{12}^*) = \Cf(f_{23}^*) \circ \Cf(f_{12}^*)$.

Of course, one can define \emph{morphisms} of graded lattices resp. sets with closure in natural ways so that $\Cf$ is actually a functor.
We will not develop these details here because this will not play any role in proving the results of this paper.

\subsection{Combinatorial geometries}
\label{subsection:combinatorial-geometries}

Let $(S,\cl)$ be a set equipped with a closure operation.
We say that $(S,\cl)$ is a {\bf combinatorial geometry} if it satisfies the following additional axioms for all $A \subset S$ and $a,b \in S$:
\begin{enumerate}
	\item {\bf Exchange:} If $a \in \cl(A \cup \{b\}) \smallsetminus \cl(A)$, then $b \in \cl(A \cup \{a\})$.
	\item {\bf Finite Character:} If $a \in \cl(A)$, then $a \in \cl(B)$ for some finite subset $B$ of $A$.
    \item {\bf Geometry:} $\cl(\varnothing) = \varnothing$ and $\cl(\{a\}) = \{a\}$.
\end{enumerate}

We will need to speak about definable sets in the first-order language of a combinatorial-geometry (or, more generally, a set with closure operation).
The language which we use is the standard one, $\Lc = (\cl_n)_{n \geq 0}$, which consists entirely of relation symbols, where $\cl_n$ is an $(n+1)$-ary relation.
For a set with closure operation $S = (S,\cl)$, we interpret $S$ as an $\Lc$-structure as follows.
The universe of a given structure is precisely the underlying set $S$.
For each $n$, the $(n+1)$-ary relation $\cl_n(x_0;x_1,\ldots,x_n)$ defines the closure $\cl(a_1,\ldots,a_n)$ by saying
\[ \cl_n(a_0;a_1,\ldots,a_n) \Longleftrightarrow a_0 \in \cl(a_1,\ldots,a_n). \]

To simplify the notation, we will use underlined boldface characters to denote (possibly empty) tuples of variables and/or constants of a given length; the length of a tuple $\tplx$ will be denoted by $l(\tplx)$.
For example, we abbreviate $(x_1,\ldots,x_r)$ as $\tplx$ and set $l(\tplx) = r$.
For $\psi(\tplx)$, a first-order formula in $\Lc$ with $l(\tplx)$ free variables, one defines:
\[ 
	\psi(S) = \{\tplc \in S^{l(\tplx)} \ : \ \psi(\tplc) \text{ holds true in } S \}. 
\]
Recall that a set $A$ is called a {\bf $\varnothing$-definable} subset of $S^r$ if there exists some first-order $\Lc$-formula, say $\psi(\tplx)$, such that $r = l(\tplx)$ and $A = \psi(S)$.
Similarly, for $\psi(\tplx;\tply)$, a first-order formula in $\Lc$ with $l(\tplx)+l(\tply)$ free variables, and $\tplb \in S^{l(\tply)}$, one defines:
\[
	\psi(S;\tplb) = \{\tplc \in S^{l(\tplx)} \ : \ \psi(\tplc;\tplb) \text{ holds true in } S \}. 
\]
Given a non-empty subset $P \subset S$, the set $A$ is called a {\bf $P$-definable} subset of $S^r$ if there exists a some first-order $\Lc$-formula, say $\psi(\tplx;\tply)$, and some $\tplb \in P^{l(\tply)}$, such that $r = l(\tplx)$ and $A = \psi(S;\tplb)$.
We say that $A$ is {\bf definable} if it is a $P$-definable subset of $S^r$ for some $r$ and for some (possibly empty) subset $P$ of $S$.
For example, if $a_1,\ldots,a_s$ are elements of $S$, then the closure $\cl(a_1,\ldots,a_s)$ is a definable subset of $S$ since:
\[ \cl(a_1,\ldots,a_s) = \cl_n(S;a_1,\ldots,a_s). \]
Note that the \emph{finite character} axiom for a combinatorial geometry ensures that the full structure of a combinatorial geometry $S$, as a set with closure operation, can be recovered from the $\Lc$-structure of $S$.

\subsection{Combinatorial geometry of algebraic closure}
\label{subsec:comb-geom-of-alg-closure}

Let $K|k$ be an extension of fields of finite transcendence degree and assume that $k$ is algebraically closed in $K$.
We denote $\Cf(\G^*(K|k))$ by $\G(K|k)$.
In other words, $\G(K|k)$ is a set with a closure operation whose underlying set is precisely $\G^1(K|k)$.

The closure operation on $\G(K|k)$ can be described explicitly as follows.
For $t \in K \smallsetminus k$, define $\kappa_t := \overline{k(t)} \cap K$.
Thus, the underlying set of $\G(K|k)$ is precisely 
\[\G^1(K|k) = \{\kappa_t \ : \ t \in K \smallsetminus k\}\]
which is the collection of relatively algebraically closed subextensions of $K|k$ of transcendence degree $1$ over $k$.
For a subset $S \subset \G^1(K|k)$, denote by $k(S)$ the compositum of the terms in $S$ (as subfields of $K$) and denote by $\kappa_S$ the field $\overline{k(S)} \cap K$, as in \S\ref{sub:algebraic_closure}.
The closure of $S$ is defined as follows:
\[ \cl(S) = \{ \kappa_t \in \G^1(K|k) \ : \ \kappa_t \subset \kappa_S \}.\]
It is fairly straightforward to check that $\G(K|k)$ is actually a \emph{combinatorial geometry}, and that $\G(K|k)$ is precisely $\Cf(\G^*(K|k))$, the set with closure operation associated to the graded lattice $\G^*(K|k)$ via the $\Cf$-construction of \S\ref{subsection:closure-operations}.

Suppose now that $K|k$ and $L|l$ are function fields over algebraically closed fields $k$ resp. $l$.
Then any isomorphism $\sigma : K^i \rightarrow L^i$ such that $\sigma k = l$ induces an isomorphism $\G(K^i|k) \rightarrow \G(L^i|l)$.
Thus, we have a canonical map $\Isom^i(K,L) \rightarrow \Isom^{\cl}(\G(K^i|k),\G(L^i|l))$.
It is clear that this map is compatible with compositions of isomorphisms (i.e. $\G$ is functorial with respect to isomorphisms).

If $\Char k = p > 0$, then the Frobenius isomorphism $\operatorname{Frob}_p : K^i \rightarrow K^i$ induces the \emph{identity} isomorphism on $\G(K^i|k)$.
Thus, we obtain an induced map 
\[ \Isom^i_F(K,L) \rightarrow \Isom^{\cl}(\G(K^i|k),\G(L^i|l)).\]
Finally, we have a canonical isomorphism of combinatorial geometries $\G(K|k) \cong \G(K^i|k)$ defined on underlying sets by sending $F \in \G^1(K|k)$ to $\overline F \cap K^i = F^i \in \G^1(K^i|k)$.
To summarize, we have a canonical map of isomorphism sets which is compatible with composition of isomorphisms:
\[ \Isom^i_F(K,L) \rightarrow \Isom^{\cl}(\G(K|k),\G(L|l)). \]

The main results concerning the classification of combinatorial geometries of algebraic closure were first developed for extensions of algebraically closed fields by {\sc Evans-Hrushovski} \cite{Evans1991}, \cite{Evans1995}.
These results were extended to extensions of arbitrary fields by {\sc Gismatullin} \cite{Gismatullin2008} using similar arguments.
These main results essentially assert that $K^i$ is encoded in $\G(K|k)$ in a first-order way, independent from $K|k$, using the language $\Lc$ from \S\ref{subsection:combinatorial-geometries}.
We summarize this theorem from loc.cit. in a special case which we will need in order to prove Theorem \ref{thm:Main-Theorem-A-milnor}.
\begin{theorem}[\cite{Gismatullin2008} Theorem 4.2]
\label{thm:fundamental-theorem-of-field-theoretic-cl-geometries}
Let $k,l$ be algebraically closed fields and let $K|k$ and $L|l$ be function fields such that $\trdeg(K|k) \geq 5$. 
Then the following hold:
\begin{enumerate}
	\item The field $K^i$ is uniformly interpretable (with parameters) in $\G(K|k)$.
	\item Any isomorphism $\sigma : \G(K|k) \rightarrow \G(L|l)$ of combinatorial geometries is induced by an isomorphism of fields $\tilde \sigma : K^i \rightarrow L^i$, which is unique up-to Frobenius twists, such that $\tilde\sigma k = l$.
	Namely, the following canonical map is a bijection:
	\[ \Isom^i_F(K,L) \rightarrow \Isom^{\cl}(\G(K|k),\G(L|l)). \] 
\end{enumerate}
\end{theorem}

We mention Theorem \ref{thm:fundamental-theorem-of-field-theoretic-cl-geometries}(1) in particular to stress that the process to reconstruct $K^i|k$ from $\Kcalm(K|k)$ in Theorem \ref{thm:Main-Theorem-A-milnor}(1) can be made very explicit (although it is not effective due to the use of Propositions \ref{prop:recovering-higher-dimensions} and \ref{prop:recovering-points}).
In simpler terms, Theorem \ref{thm:fundamental-theorem-of-field-theoretic-cl-geometries}(1) says that there are first-order formulas $P(\tply)$, $U(\tplx_1;\tply)$, $E(\tplx_2;\tply)$, $A(\tplx_3;\tply)$ and $M(\tplx_4;\tply)$ in $\Lc$, which are \emph{independent} of $K|k$, such that for all $K|k$ as in the theorem, and any \emph{arbitrary choice} of $\tplb \in P(\G(K|k))$, the following hold for $\G := \G(K|k)$:
\begin{enumerate}
	\item One has $E(\G;\tplb) \subset U(\G;\tplb)^2$, and this subset is an equivalence relation on $U(\G;\tplb)$.
	Let $\mathbf{U}$ denote the set of equivalence classes of this equivalence relation.
	\item There are binary operations $+^\G$ and $\times^\G$ on $\mathbf{U}$ such that $A(\G;\tplb)$ resp. $M(\G;\tplb)$ is the preimage of the graph of $+^\G$ resp. $\times^\G$ under the projection $U(\G,\tplb)^3 \twoheadrightarrow \mathbf{U}^3$.
	\item There is a bijection $K^i \xrightarrow{\cong} \mathbf{U}$ of sets which induces an isomorphism of structures $(K^i,+,\times) \cong (\mathbf{U},+^\G,\times^\G)$.
\end{enumerate}
In particular, $(\mathbf{U},+^\G,\times^\G)$ is a field where the addition $+^\G$ defined by $A(\G;\tplb)$, the multiplication $\times^\G$ is defined by $M(\G;\tplb)$, and this field is isomorphic to $K^i$.
Thus the field $K^i$ can be reconstructed from $\G(K|k)$ using the formulas $P,U,E,A$ and $M$ which are independent of $K|k$.
The subfield $k$ can then be recovered from $K^i$ as the collection of (multiplicatively) divisible elements of $K^i$, since $K$ is a function field over $k = \bar k$.

\section{Proof of Theorem \ref{thm:Main-Theorem-A-milnor}}
\label{section:proof-of-theorem-milnor}

We will use the notation of Theorem \ref{thm:Main-Theorem-A-milnor}.
Namely, $K|k$ and $L|l$ are function fields over algebraically closed fields such that $\Char k, \Char l \neq \ell$ and $\trdeg(K|k) \geq 5$.

\vskip 10pt
\noindent\emph{Proof of (1):}
We start with our given data $\Kcalm(K|k)$:
\begin{itemize}
	\item The groups $\k_1(K)$ and $\k_2(K)$.
	\item The multiplication map $\k_1(K) \otimes \k_1(K) \rightarrow \k_2(K)$.
	\item The collection $\Gfor(K|k)$ of rational subgroups of $\k_1(K)$.
\end{itemize}

The work required to reconstruct the field $K^i|k$ from $\Kcalm(K|k)$ has already been done, and now it's just a matter of putting everything together.
The following are the steps which reconstruct $K^i|k$ from $\Kcalm(K|k)$.
\begin{enumerate}[leftmargin=*]
	\item First, we can reconstruct the whole mod-$\ell$ Milnor K-ring $\k_*(K)$ using $\Kcalm(K|k)$, since $\k_*(K)$ is a quadratic algebra.
	More precisely, we have:
	\[ \k_*(K) = \frac{\operatorname{T}_*(\k_1(K))}{\langle \ker(\k_1(K) \otimes \k_1(K) \rightarrow \k_2(K)) \rangle} \]
	where $\operatorname{T}_*(\k_1(K))$ is the (graded) tensor-algebra generated by $\k_1(K)$ in degree $1$, and $\langle \ker(\k_1(K) \otimes \k_1(K) \rightarrow \k_2(K)) \rangle$ is the ideal of $\operatorname{T}_*(\k_1(K))$ generated by the kernel of the multiplication map $\k_1(K) \otimes \k_1(K) \rightarrow \k_2(K)$.
	In step (2) below, we will tacitly use the fact that we have constructed the full mod-$\ell$ Milnor K-ring $\k_*(K)$ of $K$.
	In particular, we are now able to compute the Milnor-dimensions of subsets of $\k_1(K)$.
	\item Recall that $\Gf^*(K|k)$ is a graded lattice by Proposition \ref{prop:alg_closure_mod_ell_lattice}.
	Using $\k_*(K)$ along with the collection $\Gfor(K|k)$, we next reconstruct this graded lattice $\Gf^*(K|k)$ using the following recipe:
	\begin{enumerate}
		\item Proposition \ref{prop:recovering-higher-dimensions} shows how to reconstruct $\Gf^r(K|k)$ for $r \geq 2$ using $\k_*(K)$ and $\Gfor(K|k)$.
		\item Having reconstructed $\Gf^r(K|k)$ for $r \geq 2$, Proposition \ref{prop:recovering-points} now shows how to reconstruct $\Gf^1(K|k)$ using $\k_*(K)$ and $\Gfor(K|k)$.
		\item The partial ordering on $\Gf^*(K|k)$ arises from inclusion of subgroups in $\k_1(K)$.
		\item The grading of $\Gf^*(K|k)$ is recovered as part of Propositions \ref{prop:recovering-higher-dimensions} and \ref{prop:recovering-points}.
		Alternatively, Lemmas \ref{lemma:milnor-dimension-of-function-fields} and \ref{lemma:independent-coords} provide the following recipe to recover the grading: For an element $A \in \Gf^*(K|k)$ one has $\dimm_K(A) = r$ if and only if $A \in \Gf^r(K|k)$. 
	\end{enumerate}
	\item Since $\Gf^*(K|k)$ is a graded lattice (Proposition \ref{prop:alg_closure_mod_ell_lattice}), we may apply the $\Cf$-construction from \S\ref{subsection:closure-operations} and obtain $\Gf(K|k) := \Cf(\Gf^*(K|k))$, the associated set with closure operation.
	On the other hand, $\Cf(\G^*(K|k)) = \G(K|k)$ is a combinatorial geometry, as described in \S\ref{subsec:comb-geom-of-alg-closure}.
	Since $\Omega_K : \G^*(K|k) \rightarrow \Gf^*(K|k)$ is an isomorphism by Proposition \ref{prop:alg_closure_mod_ell_lattice}, we see that $\Gf(K|k)$ is also a combinatorial geometry, and that $\Omega_K$ induces a canonical isomorphism $\G(K|k) \rightarrow \Gf(K|k)$ of combinatorial geometries.
	\item By Theorem \ref{thm:fundamental-theorem-of-field-theoretic-cl-geometries}(1), the field $K^i$ is uniformly interpretable in $\G(K|k)$.
	Since $\Omega_K$ induces an isomorphism $\G(K|k) \rightarrow \Gf(K|k)$ of combinatorial geometries, we see that $K^i$ is (uniformly) interpretable also in $\Gf(K|k)$ using the same interpretation/formulas used to construct $K^i$ from $\G(K|k)$.
	Finally, we can recover $k$ as the subset of multiplicatively divisible elements of $K^i$.
\end{enumerate}
This completes the proof of Theorem \ref{thm:Main-Theorem-A-milnor}(1).

\vskip 10pt
\noindent\emph{Proof of (2):}
We split up the proof of (2) into two steps.
In the first step, we trace through the steps from the proof of (1) above to construct a map
\[ \UIsomm_{\rm rat}(\k_1(K),\k_1(L)) \rightarrow \Isom^i_F(K,L) \]
which will be our candidate for the inverse of $\Isom^i_F(K,L) \rightarrow  \UIsomm_{\rm rat}(\k_1(K),\k_1(L))$.
It will be clear from the construction that this map is a \emph{left inverse}, meaning that the composition
\[ \Isom^i_F(K,L) \rightarrow \UIsomm_{\rm rat}(\k_1(K),\k_1(L)) \rightarrow \Isom^i_F(K,L) \]
is the identity.
It will also be clear from the construction that this map is compatible with compositions of isomorphisms on either side; this will allow us to reduce to the case where $K|k = L|l$.
In the second step, we complete the proof that this candidate is indeed the inverse of the canonical map $\Isom^i_F(K,L) \rightarrow  \UIsomm_{\rm rat}(\k_1(K),\k_1(L))$.

\vskip 10pt
\noindent\underline{{\bf Step 1:} Constructing the inverse.}
\vskip 10pt

First, note that if $\UIsomm_{\rm rat}(\k_1(K),\k_1(L))$ is non-empty, then the full mod-$\ell$ Milnor K-rings $\k_*(K)$ and $\k_*(L)$ are isomorphic, since $\k_*(K)$ and $\k_*(L)$ are quadratic algebras (see Step (1) in the proof of assertion (1) above).
Therefore, one has $\trdeg(K|k) = \trdeg(L|l)$, since $\trdeg(K|k) = \dimm_K(\k_1(K))$ and $\trdeg(L|l) = \dimm_L(\k_1(L))$ by Lemmas \ref{lemma:milnor-dimension-of-function-fields} and \ref{lemma:independent-coords}.
Thus, we can also apply the process from assertion (1) to $\Kcalm(L|l)$ to reconstruct $L^i|l$.
The construction of the map $\UIsomm_{\rm rat}(\k_1(K),\k_1(L)) \rightarrow \Isom^i_F(K,L)$ follows from the observation that this reconstruction process from assertion (1) is compatible with isomorphisms and invariant under multiplying isomorphisms by $\epsilon \in (\Z/\ell)^\times$.
The precise construction is outlined in the steps below.
\begin{enumerate}[leftmargin=*]
	\item Propositions \ref{prop:recovering-higher-dimensions} and \ref{prop:recovering-points} yield a canonical map:
	\[ \Isomm_{\rm rat}(\k_1(K),\k_1(L)) \rightarrow \Isom^*(\Gf^*(K|k),\Gf^*(L|l)). \]
	It is clear from the construction that this map is compatible with compositions and that it factors through $\UIsomm_{\rm rat}(\k_1(K),\k_1(L))$.
	Thus, we obtain the first component of our map:
	\[ \UIsomm_{\rm rat}(\k_1(K),\k_1(L)) \rightarrow \Isom^*(\Gf^*(K|k),\Gf^*(L|l)). \]
	\item Proposition \ref{prop:alg_closure_mod_ell_lattice} shows that $\Gf^*(K|k)$ and $\Gf^*(L|l)$ are graded lattices.
	Denoting $\Gf(K|k) := \Cf(\Gf^*(K|k))$ and $\Gf(L|l) := \Cf(\Gf^*(L|l))$ as in step (3) from the proof of (1) above, the $\Cf$-construction from \S\ref{subsection:closure-operations} yields a canonical map which is compatible with compositions:
	\[ \Isom^*(\Gf^*(K|k),\Gf^*(L|l)) \rightarrow \Isom^{\cl}(\Gf(K|k),\Gf(L|l)). \]
	\item Proposition \ref{prop:alg_closure_mod_ell_lattice} furthermore shows that $\Omega_K$ resp. $\Omega_L$ induce a canonical \emph{bijection} which is compatible with compositions:
	\[ \Isom^{\cl}(\G(K|k),\G(L|l)) \xrightarrow{\cong} \Isom^{\cl}(\Gf(K|k),\Gf(L|l)). \]
	The next component is the inverse of the map above, which is also a bijection that is compatible with compositions:
	\[ \Isom^{\cl}(\Gf(K|k),\Gf(L|l)) \xrightarrow{\cong} \Isom^{\cl}(\G(K|k),\G(L|l)). \]
	\item Theorem \ref{thm:fundamental-theorem-of-field-theoretic-cl-geometries}(2) says that the canonical map 
	\[ \Isom^i_F(K,L) \rightarrow \Isom^{\cl}(\G(K|k),\G(L|l))\]
	is a \emph{bijection}.
	Since this map is compatible with compositions, the same holds for its inverse
	\[ \Isom^{\cl}(\G(K|k),\G(L|l)) \rightarrow \Isom^i_F(K,L)\]
	which forms the last component.
\end{enumerate}
Taking the composition of the maps described above, we obtain a map 
\[ \underline{\eta_{K,L}} : \UIsomm_{\rm rat}(\k_1(K),\k_1(L)) \rightarrow \Isom^i_F(K,L) \]
which is the candidate for the inverse of $\Isom^i_F(K,L) \rightarrow \UIsomm_{\rm rat}(\k_1(K),\k_1(L))$.
Composing with the canonical map $\Isomm_{\rm rat}(\k_1(K),\k_1(L)) \rightarrow \UIsomm_{\rm rat}(\k_1(K),\k_1(L))$, we obtain another map which will be used in the second step of the proof:
\[ \eta_{K,L} : \Isomm_{\rm rat}(\k_1(K),\k_1(L)) \rightarrow \Isom^i_F(K,L). \]

\vskip 10pt
\noindent\underline{{\bf Step 2:} Completing the proof.}
\vskip 10pt

By tracing through the construction above, it is easy to see that $\underline{\eta_{K,L}}$ is a \emph{left inverse} for the canonical map $\Isom^i_F(K,L) \rightarrow \UIsomm_{\rm rat}(\k_1(K),\k_1(L))$ which was described in \S\ref{subsec:milnor-variant}. 
As noted above, we also see that $\eta_{K,L}$ is compatible with compositions of isomorphisms.
Namely, if $K|k$, $L|l$ and $M|m$ are function fields over algebraically closed fields such that $\Char k, \Char l, \Char m \neq \ell$ and $\trdeg(K|k) \geq 5$, and if $f \in \Isomm_{\rm rat}(\k_1(K),\k_1(L))$ and $g \in \Isomm_{\rm rat}(\k_1(L),\k_1(M))$ are given, then 
\[ \eta_{K,M}(g \circ f) = \eta_{L,M}(g) \circ \eta_{K,L}(f). \]

Define $\Aut_F^i(K) := \Isom^i_F(K,K)$ and $\Aut^{\rm M}_{\rm rat}(\k_1(K)) := \Isomm_{\rm rat}(\k_1(K),\k_1(K))$, both considered as groups under composition.
Thus, the construction above yields a \emph{group homomorphism}
\[ \eta_K := \eta_{K,K} : \Aut^{\rm M}_{\rm rat}(\k_1(K)) \rightarrow \Aut_F^i(K). \]
To complete the proof of Theorem \ref{thm:Main-Theorem-A-milnor}(2), it suffices to prove that the kernel of $\eta_K$ is precisely $(\Z/\ell)^\times \cdot \mathbf{1}_{\k_1(K)}$.
It is clear from the construction that $(\Z/\ell)^\times \cdot \mathbf{1}_{\k_1(K)}$ is contained in $\ker\eta_K$.
Thus, the rest of the proof will be devoted to proving the opposite inclusion.

Let $\Phi$ be an element of $\ker\eta_K$ which is going to be fixed for the rest of the proof; our goal is to show that there exists some $\epsilon \in (\Z/\ell)^\times$ such that $\Phi = \epsilon \cdot \mathbf{1}_{\k_1(K)}$.
Since $\eta_K$ factors through the \emph{group isomorphisms} furnished by Proposition \ref{prop:alg_closure_mod_ell_lattice} and Theorem \ref{thm:fundamental-theorem-of-field-theoretic-cl-geometries},
\[ \Isom^{\cl}(\Gf(K|k),\Gf(K|k)) \cong \Isom^{\cl}(\G(K|k),\G(K|k)) \cong \Aut^i_F(K), \]
we see that for all $A \in \Gf^1(K|k)$, one has $\Phi A = A$.

For a valuation $v$ of $K$, we define $\Uf_v := \Omega_K(U_v)$, the image of the $v$-units in $\k_1(K)$.
Note that the map $v : K^\times \rightarrow vK$ induces a canonical isomorphism $\k_1(K)/\Uf_v \xrightarrow{\cong} vK/\ell$.
We begin by showing that the $\Uf_v$ are $\Phi$-invariant for \emph{divisorial} valuations $v$ of $K$.

\begin{lemma}
\label{lemma:units-generated-by-subfields}
	In the notation above, one has $\Phi \Uf_v = \Uf_v$ for all divisorial valuations $v$ of $K$.
\end{lemma}
\begin{proof}
Let $v$ be a divisorial valuation of $K$.
Since $\Phi A = A$ for all $A \in \Gf^1(K|k)$, it suffices to prove that $\Uf_v$ is generated by the elements $A$ of $\Gf^1(K|k)$ such that $A \subset \Uf_v$.
Because of this, we see that it suffices to prove that $U_v$ is generated by $F^\times$ for $F \in \G^1(K|k)$ such that $F^\times \subset U_v$.

For $t \in \Oc_v$, denote by $\bar t$ the image of $t$ in $\kappa(v)$.
Since $v$ is divisorial, we know that $k \subset \Oc_v$, and thus $\kappa(v)$ is a field extension of $k$ in a canonical way.
In fact, $\kappa(v)$ is a function field over $k$ of transcendence degree $\trdeg(K|k)-1$.

If $t \in U_v$ is a $v$-unit such that $\bar t$ is transcendental over $k$, then the restriction of $v$ to $k(t)$ is trivial, since the map $\Oc_v \rightarrow \kappa(v)$ sends $k[t]$ isomorphically onto $k[\bar t]$.
Thus, the restriction of $v$ to $\overline{k(t)} \cap K =: \kappa_t \in \G^1(K|k)$ is also trivial since $\kappa_t$ is a finite extension of $k(t)$.
Therefore $\kappa_t^\times \subset U_v$.

Since $\trdeg(K|k) \geq 5$ by assumption, we have $\trdeg(\kappa(v)|k) \geq 4$. 
Thus, there exist (many) elements $t \in U_v$ such that $\bar t$ is transcendental over $k$, and we fix such an element $t_0 \in U_v$.
On the other hand, if $x \in U_v$ is such that $\bar x \in k$, then $x \cdot t_0^{-1}$ and $t_0$ both have transcendental images in $\kappa(v)$, and one has $x = (x \cdot t_0^{-1}) \cdot t_0$.
Therefore, any element of $U_v$ is a product of (at most two) elements $t \in U_v$ such that $\bar t$ is transcendental over $k$.
\end{proof}

Now let $A \in \Gf^1(K|k)$ be given.
For a divisorial valuation $v$ of $K$, let $A_v = A \cap \Uf_v$ denote the kernel of the canonical map
\[ A \hookrightarrow \k_1(K) \twoheadrightarrow \k_1(K)/\Uf_v \cong \Z/\ell. \]
We define:
\[ \Delta_A := \{A_v \ : \ A_v \neq A \}\]
where $v$ varies over the divisorial valuations of $K$.
In particular, for all $T \in \Delta_A$, one has $A/T \cong \Z/\ell$.
Furthermore, since $\Phi A = A$, and $\Phi \Uf_v = \Uf_v$ for all divisorial valuations $v$ of $K$ by Lemma \ref{lemma:units-generated-by-subfields}, we deduce that $\Phi T = T$ for all $T \in \Delta_A$.

Next, we will show that there is a natural correspondence between $\Delta_A$ and (some of) the divisorial valuations of $F$ where $F \in \G^1(K|k)$ is such that $\Omega_K(F) = A$.
\begin{lemma}
\label{lemma:reduction-divisors-to-curves}
	Let $A \in \Gf^1(K|k)$ be given and let $F \in \G^1(K|k)$ be such that $\Omega_K(F) = A$.
	Let $T$ be an element of $\Delta_A$.
	Then there exists a unique divisorial valuation $w_T$ of $F$ such that $\Omega_K(U_{w_T}) = T$.
\end{lemma}
\begin{proof}
Let $v$ be a divisorial valuation of $K$ such that $T = A_v$.
Let $w$ be the restriction of $v$ to $F$.
Since $A$ is not contained in $\Uf_v$ by the definition of $\Delta_A$, it follows that $w$ is a non-trivial valuation of $F$.
Since $v K = \Z$, we see that $w F = \Z$ as well.
Moreover, as $v$ is trivial on $k$, the same is true for $w$.
Hence, $w$ is a divisorial valuation of $F$, since $F$ is the function field of a curve over $k$.

It is clear that $\Omega_K(U_w)$ is contained in $T$.
And since $w F = \Z$, it follows that $A/\Omega_K(U_w) \cong \Z/\ell$.
Finally, since $A/T$ is also isomorphic to $\Z/\ell$, it follows that $\Omega_K(U_w) = T$.

If $w'$ is another divisorial valuation of $F$ such that $\Omega_K(U_{w'}) = T$, then $\Omega_F(U_w) = \Omega_F(U_{w'})$ since $\k_1(F) \rightarrow \k_1(K)$ is injective.
Moreover, since $\Omega_F(U_w) = \Omega_F(U_{w'})$ is a proper subgroup of $\k_1(F)$, it follows that $U_w \cdot U_{w'}$ is a \emph{proper} subgroup of $F^\times$.
In particular, $w$ and $w'$ must be \emph{dependent} as valuations of $F$, for otherwise $U_w \cdot U_{w'} = F^\times$ by the approximation theorem for independent valuations (cf. \cite{Engler2005} Theorem 2.4.1).
This implies that $w = w'$ since $w$ and $w'$ are divisorial.
Hence $w =: w_T$ is the unique divisorial valuation of $F$ such that $\Omega_K(U_w) = T$.
\end{proof}

Given a rational subgroup $A \in \Gfor(K|k)$ and $F \in \G^1(K|k)$ such that $\Omega_K(F) = A$, we say that $A$ is {\bf good} if \emph{every} divisorial valuation $w$ of $F$ is of the form $w_T$ for some $T \in \Delta_A$, where $w_T$ is as in Lemma \ref{lemma:reduction-divisors-to-curves}.
Similarly, we will say that an element $t \in K \smallsetminus k$ is {\bf very general} provided that the following two conditions hold:
\begin{enumerate}
 	\item The element $t$ is \emph{general} in $K$, in the sense of \S\ref{sec:recovering-the-mod-ell-algebraic-lattice}.
 	Recall this implies that $\Omega_K(k(t))$ is a rational subgroup of $\k_1(K)$.
 	\item The rational subgroup $\Omega_K(k(t))$ is \emph{good}, as defined above.
 \end{enumerate} 
Note that if $t$ is very general in $K$, then $t^{-1}$, $t+a$ and $t \cdot a$ are also very general in $K$ for all $a \in k^\times$.
The following lemma shows that the restriction of $\Phi$ to any \emph{good} rational subgroup of $\k_1(K)$ has the required form.
\begin{lemma}
\label{lemma:restriction-to-rational-is-epsilon}
	Let $A$ be a good rational subgroup of $\k_1(K)$.
	Then there exists an $\epsilon_A \in (\Z/\ell)^\times$, possibly dependent on $A$, such that the restriction of $\Phi$ to $A$ is precisely $\epsilon_A \cdot \mathbf{1}_A$.
\end{lemma}
\begin{proof}
Let $F \in \G^1(K|k)$ be such that $\Omega_K(F) = A$; note that the canonical map $F^\times \rightarrow \k_1(K)$, whose image is $A$, induces an isomorphism $\k_1(F) \xrightarrow{\cong} A$.
Let $\Df_F$ denote the collection of divisorial valuations of $F$.
Since $F$ is rational over $k$, the multiplicative group $F^\times/k^\times$ fits into the following short exact sequence:
\[ 0 \rightarrow F^\times/k^\times \xrightarrow{\divv} \bigoplus_{w \in \Df_F} \Z \cdot [w] \xrightarrow{\text{sum}} \Z \rightarrow 0, \]
where $\divv(x) = \sum_{w \in \Df_F} w(x) \cdot [w]$ is the usual divisor map.
Since $k$ is algebraically closed (hence $k^\times$ is divisible), we see that $(F^\times/k^\times) \otimes \Z/\ell = \k_1(F)$.
Since $\Z$ is torsion-free, we therefore obtain the following short exact sequence, by tensoring the above exact sequence with $\Z/\ell$:
\[ 0 \rightarrow \k_1(F) \xrightarrow{\divv} \bigoplus_{w \in \Df_F} \Z/\ell \cdot [w] \xrightarrow{\text{sum}} \Z/\ell \rightarrow 0. \]

Since $A$ is a good rational subgroup, by Lemma \ref{lemma:reduction-divisors-to-curves}, the map $w \mapsto \Omega_K(U_w)$ is a \emph{bijection} between $\Df_F$ and $\Delta_A$.
Moreover, if $w \in \Df_F$ is a divisorial valuation of $F$ and $T := \Omega_K(U_w) \in \Delta_A$, then there exists an isomorphism $\Psi_T : A/T \rightarrow \Z/\ell$ such that the following diagram commutes:
\[
\xymatrix{
	F^\times \ar@{->>}[d] \ar@{->>}[r]^{w} & \Z \ar@{->>}[r]^{\rm canon.} & \Z/\ell \ar@{=}[d]\\
	A \ar@{->>}[r]_{\rm canon.} & A/T \ar[r]_{\Psi_T} & \Z/\ell
}
\]
In particular, every element of $A$ is contained in all but at most finitely many $T \in \Delta_A$.
Thus, we have a canonical map $\divv_A : A \rightarrow \bigoplus_{T \in \Delta_A} A/T$.
Moreover, combining the isomorphisms $\Psi_T$ for $T \in \Delta_A$, and identifying $\Delta_A$ and $\Df_F$ via the bijection obtained from Lemma \ref{lemma:reduction-divisors-to-curves}, we obtain an isomorphism 
\[ \Psi = (\Psi_T)_T : \bigoplus_{T \in \Delta_A} A/T \rightarrow \bigoplus_{T \in \Delta_A} \Z/\ell \cdot [T] \] 
which makes the following diagram commute:
\[\xymatrix{
\k_1(F) \ar[d]_{\cong} \ar[r]^-{\divv} & \bigoplus_{w \in \Df_F} \Z/\ell \cdot [w] \ar@{=}[r] & \bigoplus_{T \in \Delta_A} \Z/\ell \cdot [T] \\
A \ar[rr]_-{\divv_A} & & \bigoplus_{T \in \Delta_A} A/T \ar[u]_{\Psi}
}\]
We let $\divv_\ell := \Psi \circ \divv_A$, so that $\divv_\ell$ fits into the following short exact sequence:
\[ 0 \rightarrow A \xrightarrow{\divv_\ell} \bigoplus_{T \in \Delta_A} \Z/\ell \cdot [T] \xrightarrow{\text{sum}} \Z/\ell \rightarrow 0. \]

Next, since $\Phi A = A$ and $\Phi T = T$ for all $T \in \Delta_A$, it follows that $\Phi$ induces automorphisms $\Phi_T : A / T \rightarrow A/T$ for each $T \in \Delta_A$, such that $\divv_A \circ \Phi = (\Phi_T)_T \circ \divv_A$.
Moreover, for every $T \in \Delta_A$, there exists some $\epsilon_T \in (\Z/\ell)^\times = \Aut(\Z/\ell)$ such that $\Psi_T \circ \Phi_T = \epsilon_T \cdot \Psi_T$.
In particular, these maps all fit into the following commutative diagram with exact rows:
\[\xymatrix{
0 \ar[r] & A\ar[d]_{\Phi} \ar[r]^-{\divv_\ell} & \bigoplus_{T \in \Delta_A} \Z/\ell \cdot [T] \ar[d]^{(\epsilon_T)_T} \ar[r]^-{\text{sum}} & \Z/\ell \ar@{=}[d]\ar[r] & 0 \\
0 \ar[r] & A \ar[r]_-{\divv_\ell} & \bigoplus_{T \in \Delta_A} \Z/\ell \cdot [T] \ar[r]_-{\text{sum}} & \Z/\ell \ar[r] & 0
}\]

Let $T_1,T_2 \in \Delta_A$ be given.
Then there exists some $a \in A$ such that $\divv_\ell(a) = [T_1]-[T_2]$.
But then $\divv_\ell(\Phi(a)) = \epsilon_{T_1} \cdot [T_1] - \epsilon_{T_2} \cdot [T_2]$, and this implies that $\epsilon_{T_1} = \epsilon_{T_2}$ by the exactness of the bottom row in the diagram above.
Namely, there exists a \emph{single} $\epsilon_A \in (\Z/\ell)^\times$ such that $\Psi_T \circ \Phi_T = \epsilon_A \cdot \Phi_T$ for all $T \in \Delta_A$.
Thus, the injectivity of $\divv_\ell$ in the diagram above implies that the restriction of $\Phi$ to $A$ is precisely multiplication by this $\epsilon_A$.
\end{proof}

The conclusion of the proof of Theorem \ref{thm:Main-Theorem-A-milnor} will rely on the following refinement of Fact \ref{fact:birational-bertini}, which can be seen as a ``Birational Bertini Theorem'' for very general elements.
\begin{proposition}
\label{prop:very-general-bertini}
	In the above notation, let $t_1,t_2$ be elements of $K$ which are algebraically independent over $k$ such that $t_2$ is separable in $K$.
	Then, for all but finitely many $b \in k$, the following holds: For all but finitely many $a \in k$, the element $(t_1+a)/(t_2+b)$ is very general in $K$.
\end{proposition}
\begin{proof}
Extend $t_1,t_2$ to a transcendence base $\tbf = (t_1,\ldots,t_d)$ for $K|k$ and consider $X_0 = \A^n_{k,\tbf}$, affine $n$-space over $k$ with coordinates $\tbf$.
Given $a,b \in k$, we put $t_{a,b} := (t_1+a)/(t_2+b)$.
The inclusion $k(t_{a,b}) \hookrightarrow k(\tbf)$ induces a dominant rational map $\pi_{a,b} : X_0 \rightarrow \Pbb^1_k$, defined on $k$-points by $(a_1,\ldots,a_n) \mapsto (a_1+a)/(a_2+b)$.
The inclusion of function fields induced by $\pi_{a,b}$ is the map $k(t) \hookrightarrow k(\tbf)$ sending $t$ to $t_{a,b}$, and we will identify $k(t_{a,b})$ with the function field of $\Pbb^1_k$.
The fibers of $X_0 \rightarrow \Pbb^1_k$ over the closed points of $\Pbb^1_k$ are of the following form:
\begin{enumerate}
	\item For $c \in \A^1_k(k)$, the fiber over $c$ is the linear subvariety of $X_0$ defined by the equation $(t_1+a) = c \cdot (t_2+b)$.
	\item For $\infty \in \Pbb^1_k(k) \smallsetminus \A^1_k(k)$, the fiber over $\infty$ is the linear subvariety of $X_0$ defined by the equation $t_2+b = 0$.
\end{enumerate}
In particular, the map $\pi_{a,b}$ is surjective.

By Fact \ref{fact:birational-bertini}, we see that for any given $b \in k$, the element $t_{a,b} = t_1/(t_2+b)+a/(t_2+b)$ is general in $K$ for all but finitely many $a \in k$.
Also, Lemma \ref{lemma:mod-ell-ramification} says that all but finitely many prime-divisors of $X_0$ are $\ell$-unramified in $K$.
Thus, by the discussion above, for all but finitely many $b \in k$, the following two conditions hold for all but finitely many $a \in k$ (the excluded $a$'s possibly depending on $b$): 
\begin{enumerate}
	\item The element $t_{a,b}$ is general in $K$.
	\item For any closed point $x$ of $\Pbb^1_k$, the fiber of the map $\pi_{a,b}$ above $x$ is $\ell$-unramified in $K$.
\end{enumerate}
To conclude the proof, we will show that any $t_{a,b}$ satisfying these two conditions is actually very general in $K$.
By condition (1), we know that $t_{a,b}$ is already general in $K$.
Thus, setting $A := \Omega_K(k(t_{a,b}))$, we must show that every divisorial valuation $w$ of $k(t_{a,b})$ is of the form $w_T$ for some $T \in \Delta_A$, where $w_T$ is as described in Lemma \ref{lemma:reduction-divisors-to-curves}.

Let $w$ be a divisorial valuation of $k(t_{a,b})$ and let $x$ be the unique closed point (i.e. prime-divisor) of $\Pbb^1_k$ such that $w = w_x$.
Let $P_0$ be the generic point of $(X_0)_x$, the fiber of $\pi_{a,b}$ over $x$, and let $v_0 = v_{P_0}$ be the valuation of $k(\tbf)$ associated to $P_0$.
Then $w$ is the restriction of $v_0$ to $k(t_{a,b})$, and $v_0 k(\tbf) = w k(t_{a,b})$.

Let $X$ denote the normalization of $X_0$ in $K$, and let $P$ be a prime divisor on $X$ which prolongs $P_0$.
Furthermore, let $v = v_P$ denote the valuation of $K$ associated to $P$.
Then $w$ is the restriction of $v$ to $k(t_{a,b})$ and $v_0$ is the restriction of $w$ to $k(\tbf)$.
Furthermore, as $P_0$ is $\ell$-unramified in $K$, we see that $v_0k(\tbf)$ is not contained in $\ell \cdot v K$, and therefore $w k(t_{a,b})$ is not contained in $\ell \cdot v K$. 
In particular, $\Omega_K(k(t_{a,b})) = A$ is not contained in $\Uf_v$.
Thus, $\Uf_v \cap A =: T$ is an element of $\Delta_A$.
Finally, the argument of Lemma \ref{lemma:reduction-divisors-to-curves} shows that $w = w_T$.
Namely, it is clear that $\Omega_K(U_w) \subset T$, while both $A/\Omega_K(U_w)$ and $A/T$ are isomorphic to $\Z/\ell$; hence $\Omega_K(U_w) = T$.
\end{proof}

By Lemma \ref{lemma:restriction-to-rational-is-epsilon}, in order to conclude the proof, we must prove the following two assertions:
\begin{enumerate}
	\item That $\k_1(K)$ is generated by its good rational subgroups.
	\item That for any two good rational subgroups $A,B$ of $\k_1(K)$, one has $\epsilon_A = \epsilon_B$, where $\epsilon_A, \epsilon_B \in (\Z/\ell)^\times$ are as in the statement of Lemma \ref{lemma:restriction-to-rational-is-epsilon}.
\end{enumerate}

For assertion (1), assume that $x \in K^\times \smallsetminus k^\times$ is separable in $K$.
Let $y$ be an element of $K^\times$ which is algebraically independent from $x$.
By Proposition \ref{prop:very-general-bertini}, there exist $a,b \in k$ such that both $(xy-ab)/(x+b)$ and $(y+a)/(x+b)$ are very general in $K$.
Therefore, $(xy+ax)/(x+b) = (xy-ab)/(x+b)+a$ and $(x+b)/(y+a) = ((y+a)/(x+b))^{-1}$ are also very general in $K$.
Hence 
\[ x = \left(\frac{xy+ax}{x+b}\right) \cdot \left(\frac{x+b}{y+a}\right)
\]
is a product of two very general elements of $K$.
Since any non-constant element of $K$ is a power of an element of $K$ which is separable in $K$, we deduce that $\k_1(K)$ is generated by its good rational subgroups.

For assertion (2), assume that $A$ and $B$ are two \emph{different} good rational subgroups of $\k_1(K)$, and that $x,y$ are very general elements of $K$ such that $A = \Omega_K(k(x))$ and $\Omega_K(k(y)) = B$.
Note that $x,y$ must be algebraically independent over $k$.
By Proposition \ref{prop:very-general-bertini}, there exists $b \in k$ such that for all but finitely many $a \in k$, the element $(x+a)/(y+b)$ is also very general in $K$.
Since $\k_1(k(x)) \rightarrow \k_1(K)$ is injective, we can choose an $a \in k$ such that $(x+a)/(y+b) =: z$ is very general, and such that $(x+a)$ and $(y+b)$ have $\Z/\ell$-independent images in $\k_1(K)$.
Let $C = \Omega_K(k(z))$, and let $\epsilon_A,\epsilon_B,\epsilon_C \in (\Z/\ell)^\times$ be as in Lemma \ref{lemma:restriction-to-rational-is-epsilon}.
Now we calculate in $\k_1(K)$:
\[
	\frac{(x+a)^{\epsilon_A}}{(y+b)^{\epsilon_B}} = \frac{\Phi(x+a)}{\Phi(y+b)} = \Phi\left(\frac{x+a}{y+b}\right) = \left(\frac{x+a}{y+b}\right)^{\epsilon_C} = \frac{(x+a)^{\epsilon_C}}{(y+b)^{\epsilon_C}}. 
\]
Since $(x+a)$ and $(y+b)$ have independent images in $\k_1(K)$, we deduce that $\epsilon_A = \epsilon_C = \epsilon_B$, as required.
This completes the proof of Theorem \ref{thm:Main-Theorem-A-milnor}.

\section{The Zassenhauss Filtration and Galois Cohomology}
\label{sec:zassenhauss}

In this section we recall the cohomological framework which allows us to translate between The ``Galois variant'' of our Main Theorem (Theorem \ref{thm:Main-Theorem-A}) and the ``Milnor variant'' of our Main Theorem (Theorem \ref{thm:Main-Theorem-A-milnor}).
The results in this section are by no means new -- they are merely a precise formulation of the well-known fact that the cup-product in mod-$\ell$ cohomology is ``dual'' to the commutator in the mod-$\ell$ central descending series, combined with the Merkurjev-Suslin Theorem \cite{Merkurjev1982}.
The essential calculations concerning commutators and cup products were first carried out by Labute \cite{Labute1967} (see also the exposition in \cite{Neukirch2008} \S3.9).
These calculations have seen a recent resurgence in \cite{Chebolu2009}, \cite{Efrat2011b}, \cite{Efrat2011c}, \cite{2013arXiv1310.5613T}, especially in connection with the Merkurjev-Suslin Theorem \cite{Merkurjev1982} and the recent proof of the Bloch-Kato conjecture due to Voevodsky-Rost et al. \cite{Voevodsky2011}, \cite{Rost-chain-lemma}, \cite{zbMATH05594008}.
Nevertheless, the Merkurjev-Suslin Theorem \cite{Merkurjev1982} suffices for the considerations in this section, and we summarize the appropriate cohomological and group-theoretical calculations below.

Let $\Gc$ be a pro-$\ell$ group.
We will denote the mod-$\ell$ cohomology $\H^*(\Gc,\Z/\ell)$ of $\Gc$ simply by $\H^*(\Gc)$.
Recall that the Bockstein morphism 
\[ \beta : \H^1(\Gc) \rightarrow \H^2(\Gc) \]
is defined to be the connecting homomorphism associated to the short exact sequence of coefficient modules $0 \rightarrow \Z/\ell \rightarrow \Z/\ell^2 \rightarrow \Z/\ell \rightarrow 0$.

We recall the first two terms in the mod-$\ell$ Zassenhauss filtration of $\Gc$:
\begin{enumerate}
	\item $\Gc^{(2)} := [\Gc,\Gc]\cdot\Gc^\ell$.
	\item If $\ell \neq 2$ then $\Gc^{(3)} := [\Gc,\Gc^{(2)}] \cdot \Gc^\ell$.
	\item If $\ell = 2$ then $\Gc^{(3)} := [\Gc,\Gc^{(2)}] \cdot (\Gc^{(2)})^\ell$. 
\end{enumerate}
We denote the quotient $\Gc/\Gc^{(2)}$ by $\Gc^a$ and the quotient $\Gc/\Gc^{(3)}$ by $\Gc^c$.
In particular, note that $\Gc^a$ is an $\ell$-elementary abelian pro-$\ell$ group, and that $\Gc^c \rightarrow \Gc^a$ is a central extension whose kernel is also $\ell$-elementary abelian.

For $\sigma,\tau \in \Gc^a$, we define $[\sigma,\tau] := \tilde\sigma^{-1} \tilde\tau^{-1} \tilde\sigma\tilde\tau$ where $\tilde\sigma,\tilde\tau \in \Gc^c$ are some lifts of $\sigma,\tau \in \Gc^a$.
Since $\Gc^c \rightarrow \Gc^a$ is a central extension, the element $[\sigma,\tau] \in \Gc^{(2)}/\Gc^{(3)}$ doesn't depend on the choice of lifts $\tilde\sigma,\tilde\tau$.

For a discrete $\Z/\ell$-module $M$, define $\wedge^2(M) := {M \otimes M}/\langle x \otimes x \ : \ x \in M \rangle$.
For a pro-$\ell$ $\Z/\ell$-module $\Gc^a$, we define:
\[ \widehat \wedge^2(\Gc^a) := \varprojlim_{H} \wedge^2(\Gc^a/H) \]
where $H$ varies over the open subgroups of $\Gc^a$ and thus $\Gc^a/H$ is a finite (hence discrete) $\Z/\ell$-module.

For a pro-$\ell$ $\Z/\ell$-module $\Gc^a$, let $(\Gc^a)^\vee$ denote the Pontryagin-dual of $\Gc^a$.
Namely, $(\Gc^a)^\vee := \Hom(\Gc^a,\Z/\ell)$ is a discrete $\Z/\ell$-module.
Thus, the canonical perfect pairing $\Gc^a \times (\Gc^a)^\vee \rightarrow \Z/\ell$ induces a perfect pairing of wedge-products:
\[ \widehat \wedge^2(\Gc^a) \times \wedge^2((\Gc^a)^\vee) \rightarrow \Z/\ell \]
which is defined in the usual way by $(\sigma \wedge \tau) \times (f \wedge g) \mapsto f(\sigma) \cdot g(\tau) - f(\tau) \cdot g(\sigma)$.

\subsection{A presentation of $\H^2$}
\label{sub:presentation-of-H2}
Since $\Gc^a$ is isomorphic to a direct power of $\Z/\ell$, the K\"unneth formula for (profinite) group cohomology, together with the well-known structure of the cohomology ring $\H^*(\Z/\ell,\Z/\ell)$, yield the following fact.
See \cite{Efrat2011b} \S2.D, and/or \cite{2013arXiv1310.5613T} Lemma 2.1 for more details.
\begin{fact}
\label{fact:H2-presentation}
	In the notation above, the canonical map 
	\[
		\xymatrix@R-2pc{
		(\H^1(\Gc^a) \otimes \H^1(\Gc^a)) \oplus \H^1(\Gc^a) \ar[r] & \H^2(\Gc^a) \\
		(x \otimes y) \oplus z \ar@{|->}[r] & x \cup y + \beta z
		}
	\]
	is surjective, and the kernel of this map is generated by all elements of the form $(x \otimes x) \oplus {\ell \choose 2} \cdot x$ for $x$ varying in $\H^1(\Gc^a)$.
\end{fact}
Thus, if $\ell \neq 2$, then the map $\wedge^2(\H^1(\Gc^a)) \oplus \H^1(\Gc^a) \rightarrow \H^2(\Gc^a)$, defined by $(x \wedge y) \oplus z \mapsto x \cup y + \beta z$, is an isomorphism.
On the other hand, if $\ell = 2$, then the cup product map yields an isomorphism $\Sym^2(\H^1(\Gc^a)) \rightarrow \H^2(\Gc^a)$, where $\Sym^2$ denotes the second symmetric power.

The decomposable part of $\H^2(\Gc^a)$, denoted $\H^2(\Gc^a)_{\rm dec}$, is defined as the image of the cup product $\H^1(\Gc^a) \otimes \H^1(\Gc^a) \rightarrow \H^2(\Gc^a)$.
Thus, if $\ell \neq 2$, then the cup product yields an isomorphism $\wedge^2(\H^1(\Gc^a)) \rightarrow \H^2(\Gc^a)_{\rm dec}$.
On the other hand, if $\ell = 2$, then $\H^2(\Gc^a) = \H^2(\Gc^a)_{\rm dec}$.

\subsection{Cup-products and the Zassenhauss filtration}
\label{sub:zassen-and-cup-products}

Since the inflation map $\H^1(\Gc^a) \rightarrow \H^1(\Gc)$ is an isomorphism, we see that the Lyndon-Hochschild-Serre spectral sequence associated to the extension $\Gc \rightarrow \Gc^a$ yields an exact sequence:
\begin{align}
\label{align:LHS}
0 \rightarrow \H^1(\Gc^{(2)})^\Gc \xrightarrow{d_2} \H^2(\Gc^a) \xrightarrow{\inf} \H^2(\Gc),
\end{align}
where $d_2 := d_2^{0,1}$ is the associated differential on the $E_2$-page of the spectral sequence.
The map $d_2$ above also agrees with the \emph{transgression map} (cf. \cite{Neukirch2008} Theorem 2.4.3).
Also, note that $\H^1(\Gc^{(2)})^\Gc = \Hom_{\Gc}(\Gc^{(2)},\Z/\ell)$ is the set of $\Gc$-equivariant homomorphisms $\Gc^{(2)} \rightarrow \Z/\ell$, where $\Gc$ acts on $\Gc^{(2)}$ by conjugation and trivially on $\Z/\ell$.
Since $\Gc^c \rightarrow \Gc^a$ is a central extension with kernel $\Gc^{(2)}/\Gc^{(3)}$ which is killed by $\ell$, we obtain a canonical \emph{injective} map:
\[ (\Gc^{(2)}/\Gc^{(3)})^\vee \hookrightarrow \H^1(\Gc^{(2)})^\Gc. \]
The following lemma describes the image of this injection with respect to $d_2$.
This lemma is essentially identical to \cite{Efrat2011c} Corollary 10.9, but we still provide a detailed proof to keep the discussion as self-contained as possible.

\begin{lemma}
\label{lemma:pre-image-of-H2-dec}
The image of the canonical injective map $(\Gc^{(2)}/\Gc^{(3)})^\vee \hookrightarrow \H^1(\Gc^{(2)})^\Gc$ is precisely $d_2^{-1}(\H^2(\Gc^a)_{\rm dec})$, the preimage of the decomposable part of $\H^2(\Gc^a)$ under the map $d_2 : \H^1(\Gc^{(2)})^\Gc \rightarrow \H^2(\Gc^a)$.
\end{lemma}
\begin{proof}
Let $\Gc^{[3]} := [\Gc,\Gc^{(2)}] \cdot (\Gc^{(2)})^\ell$ denote the third term in the mod-$\ell$ central descending series.
I.e. $\Gc^{[3]}$ is the left-kernel of the canonical pairing $\Gc^{(2)} \times \H^1(\Gc^{(2)})^\Gc \rightarrow \Z/\ell$ (see e.g. \cite{Efrat2011b} \S2.A).
In other words, we may identify the following three groups:
\[ \H^1(\Gc^{(2)})^\Gc = \Hom_\Gc(\Gc^{(2)},\Z/\ell) = (\Gc^{(2)}/\Gc^{[3]})^\vee.\]
Moreover, note that one has $\Gc^{[3]} \subset \Gc^{(3)}$.
If $\ell \neq 2$, then $\Gc^{(3)}$ is (topologically) generated by $\Gc^{[3]}$ and $\Gc^\ell$.
On the other hand, if $\ell = 2$, then $\Gc^{[3]} = \Gc^{(3)}$ by definition.

\vskip 5pt
\noindent\emph{Case $\ell = 2$:} 
In this case, one has $\Gc^{[3]} = \Gc^{(3)}$ and thus $(\Gc^{(2)}/\Gc^{(3)})^\vee = \H^1(\Gc^{(2)})^\Gc$.
Also, Fact \ref{fact:H2-presentation} implies that $\H^2(\Gc^a)_{\rm dec} = \H^2(\Gc^a)$.
Thus, the lemma follows trivially.

\vskip 5pt
\noindent\emph{Case $\ell \neq 2$:} 
Fact \ref{fact:H2-presentation} implies that the map $\wedge^2(\H^1(\Gc^a)) \oplus \H^1(\Gc^a) \rightarrow \H^2(\Gc^a)$, given by $(x \wedge y) \oplus z \mapsto x \cup y + \beta z$, is an isomorphism.
In particular, any element $\eta \in \H^2(\Gc^a)$ can be represented as
\[ \eta = \left(\sum_i x_i \cup y_i\right) + \beta z \]
for some $x_i,y_i,z \in \H^1(\Gc^a)$, with $z$ uniquely determined by $\eta$.
Also, note that $\eta \in \H^2(\Gc^a)_{\rm dec}$ if and only if $z = 0$.

Next, recall that $\Gc^{(3)}$ is (topologically) generated by $\Gc^{[3]}$ and $\Gc^\ell$.
Thus, the image of the injection $(\Gc^{(2)}/\Gc^{(3)})^\vee \hookrightarrow \H^1(\Gc^{(2)})^\Gc$ consists precisely of the elements 
\[ f \in \H^1(\Gc^{(2)})^\Gc = \Hom_\Gc(\Gc^{(2)},\Z/\ell) = (\Gc^{(2)}/\Gc^{[3]})^\vee \]
such that $f(\tilde\sigma^\ell) = 0$ for all $\tilde\sigma \in \Gc$.

Now let $f \in \H^1(\Gc^{(2)})^\Gc$ be given.
By the observation above using Fact \ref{fact:H2-presentation}, there exists a representation of $d_2 f$ of the form 
\[ d_2 f = \left(\sum_i x_i \cup y_i\right) + \beta z \]
for some $x_i,y_i,z \in \H^1(\Gc^a) = \Hom(\Gc^a,\Z/\ell)$, with $z$ uniquely determined by $d_2 f$.
Let $\tilde\sigma \in \Gc$ be given with image $\sigma \in \Gc^a$.
Then \cite{Neukirch2008} Proposition 3.9.14 (see also \cite{2013arXiv1310.5613T} Theorem 2) implies that $-f(\tilde\sigma^\ell) = z(\sigma)$.
Thus, the description above shows that $f$ is in the image of $(\Gc^{(2)}/\Gc^{(3)})^\vee \hookrightarrow \H^1(\Gc^{(2)})^\Gc$ if and only if $z(\sigma) = 0$ for all $\sigma \in \Gc^a$, i.e. $z = 0$ in $\H^1(\Gc^a)$.
As discussed above, this is further equivalent to $d_2 f \in \H^2(\Gc^a)_{\rm dec}$, as required.
\end{proof}

We now recall the well-known duality between the map $[\bullet,\bullet] : \widehat\wedge^2(\Gc^a) \rightarrow \Gc^{(2)}/\Gc^{(3)}$ and cup-products in mod-$\ell$ cohomology.
By Fact \ref{fact:H2-presentation}, we note that the cup product induces a canonical isomorphism
\[ \wedge^2(\H^1(\Gc^a)) \xrightarrow{\cong} \H^2(\Gc^a)_{\rm dec}/(\beta\H^1(\Gc^a) \cap \H^2(\Gc^a)_{\rm dec}). \]
We will henceforth identify these two groups via this isomorphism.
Next, note that the map $d_2 : \H^1(\Gc^{(2)})^\Gc \rightarrow \H^2(\Gc^a)$ restricts to a map $d_2 : (\Gc^{(2)}/\Gc^{(3)})^\vee \rightarrow \H^2(\Gc^a)_{\rm dec}$ by Lemma \ref{lemma:pre-image-of-H2-dec}. 
Therefore, we obtain an induced map 
\[ \Upsilon : (\Gc^{(2)}/\Gc^{(3)})^\vee \xrightarrow{-d_2} \H^2(\Gc^a)_{\rm dec} \twoheadrightarrow \wedge^2(\H^1(\Gc^a)) \]
where the surjective map on the right is induced by the identification we made above.
Note that we have replaced $d_2$ with $-d_2$ in the definition of $\Upsilon$, in order to account for the negative sign that appears in the results referenced below.
With this notation, the following fact follows immediately from \cite{Neukirch2008} Proposition 3.9.13 and/or \cite{2013arXiv1310.5613T} Theorem 2(1).
\begin{fact}
\label{fact: cup dual to commutator}
	The map $\Upsilon : (\Gc^{(2)}/\Gc^{(3)})^\vee \rightarrow \wedge^2(\H^1(\Gc^a))$ defined above is precisely the dual of the commutator map $[\bullet,\bullet] : \widehat\wedge^2(\Gc^a) \rightarrow \Gc^{(2)}/\Gc^{(3)}$.
\end{fact}

\subsection{Galois cohomology and Kummer theory}
\label{sub:Kummer-theory}

Suppose that $K$ is a field such that $\Char K \neq \ell$ and $\mu_{\ell^2} \subset K$.
Recall that $K(\ell)$ denotes the maximal pro-$\ell$ Galois extension of $K$ and that $\Gc_K := \Gal(K(\ell)|K)$ denotes the maximal pro-$\ell$ Galois group of $K$.
Furthermore, recall that Kummer theory yields a canonical perfect pairing:
\[ \Gc_K^a \times \k_1(K) \rightarrow \mu_\ell, \]
which is defined by $(\sigma,x) \mapsto \sigma \sqrt[\ell]{x}/\sqrt[\ell]{x}$.

Choose a primitive $\ell$-th root of unity $\omega \in \mu_\ell \subset K$, and consider the induced isomorphism $\mu_\ell \cong \Z/\ell$ given by $\omega^i \mapsto i$.
We therefore obtain an induced perfect pairing:
\[ (\bullet,\bullet)_\omega : \Gc_K^a \times \k_1(K) \rightarrow \Z/\ell \]
where $(\sigma,x)_\omega = i$ if and only if $\sigma\sqrt[\ell]{x}/\sqrt[\ell]{x} = \omega^i$.
In particular, the map $\k_1(K) \rightarrow \H^1(\Gc_K^a)$, defined by $x \mapsto (\bullet,x)_\omega \in \H^1(\Gc_K^a)$, is an isomorphism of discrete $\Z/\ell$-modules.

Recall that the Merkurjev-Suslin theorem \cite{Merkurjev1982} states that the canonical map $\k_2(K) \rightarrow \H^2(K,\mu_\ell^{\otimes 2})$ is an isomorphism.
In particular, the cup-product map $\H^1(K,\mu_\ell) \otimes \H^1(K,\mu_\ell) \rightarrow \H^2(K,\mu_\ell^{\otimes 2})$ is surjective.
Our choice of isomorphism $\mu_\ell \cong \Z/\ell$ given by $\omega$ induces a compatible isomorphism $\mu_\ell^{\otimes 2} \cong \Z/\ell$.
In particular, the cup-product map $\H^1(G_K) \otimes \H^1(G_K) \rightarrow \H^2(G_K)$ is surjective as well.
On the other hand, the inflation map $\H^1(\Gc_K^a) \rightarrow \H^1(G_K)$ is an isomorphism.
From this, it follows that the inflation map $\H^2(\Gc_K^a) \rightarrow \H^2(G_K)$ is surjective.
Hence the inflation map $\H^2(\Gc_K) \rightarrow \H^2(G_K)$ is surjective as well.
This observation allows us to work with $\Gc_K$ instead of $G_K$ via the following fact.
See also \cite{Chebolu2009} Remark 8.2.

\begin{fact}
\label{fact:H2-Galois-isom}
In the context above, the inflation map $\H^2(\Gc_K) \rightarrow \H^2(G_K)$ is an isomorphism.
\end{fact}
\begin{proof}
Surjectivity of this map was discussed above.
Since $K(\ell)^\times$ is $\ell$-divisible, it follows that $\H^1(G_{K(\ell)}) \cong \H^1(K(\ell),\mu_\ell) = \k_1(K(\ell)) = 0$.
Thus the map $\H^2(\Gc_K) \rightarrow \H^2(G_K)$ is injective since its kernel is the image of $d_2 : \H^1(G_{K(\ell)})^{\Gc_K} \rightarrow \H^2(\Gc_K)$.
\end{proof}

For the rest of the section, we will (canonically) identify $\H^2(\Gc_K)$ and $\H^2(G_K)$ using Fact \ref{fact:H2-Galois-isom}.
Next we discuss the Bockstein morphism in Galois cohomology.
The following lemma shows the triviality of the Bockstein morphism in the presence of sufficiently many roots of unity.
It is possible to deduce the following lemma using the description of the Bockstein morphism in Galois cohomology as a cup-product; see for example \cite{Efrat2011b} Proposition 2.6 and/or \cite{2013arXiv1310.5613T} Lemma 4.1.
However, the elegant proof that we give below was graciously suggested to us by the referee.

\begin{lemma}
\label{lemma:bockstein-is-trivial-galois}
Let $K$ be a field of characteristic $\neq \ell$ such that $\mu_{\ell^2} \subset K$.
Then the Bockstein morphism $\beta : \H^1(\Gc_K) \rightarrow \H^2(\Gc_K)$ is trivial.
\end{lemma}
\begin{proof}
By choosing compatible isomorphisms $\mu_\ell \cong \Z/\ell$ and $\mu_{\ell^2} \cong \Z/\ell^2$ of $G_K$-modules, Kummer theory shows that the map 
\[ \H^1(\Gc_K,\Z/\ell^2) \rightarrow \H^1(\Gc_K,\Z/\ell) \]
is surjective as it corresponds to the projection $K^\times/\ell^2 \rightarrow K^\times/\ell$.
The mere definition of the Bockstein morphism $\beta : \H^1(\Gc_K) \rightarrow \H^1(\Gc_K)$ as a connecting morphism shows that the image of the (surjective) map $\H^1(\Gc_K,\Z/\ell^2) \rightarrow \H^1(\Gc_K)$ is the kernel of $\beta$.
\end{proof}

Recall that the cup product induces a canonical isomorphism, as noted above
\[ \wedge^2(\H^1(\Gc_K^a)) = \H^2(\Gc_K^a)_{\rm dec}/(\beta\H^1(\Gc_K^a) \cap \H^2(\Gc_K^a)_{\rm dec}). \]
Therefore, if $\mu_{\ell^2} \subset K$, we deduce from Lemma \ref{lemma:bockstein-is-trivial-galois} that the inflation map $\H^2(\Gc_K^a)_{\rm dec} \rightarrow \H^2(\Gc_K)$ factors through $\wedge^2(\H^1(\Gc_K^a))$ and that the induced map $\wedge^2(\H^1(\Gc_K^a)) \rightarrow \H^2(\Gc_K)$ is the inflation composed with the cup-product:
\[ \wedge^2(\H^1(\Gc_K^a)) \xrightarrow{\inf} \wedge^2(\H^1(\Gc_K)) \xrightarrow{\cup} \H^2(\Gc_K). \]

We note that the perfect pairing $(\bullet,\bullet)_\omega : \Gc_K^a \times \k_1(K) \rightarrow \Z/\ell$ induces a perfect pairing on wedge-products:
\[ (\bullet,\bullet)_\omega : \widehat\wedge^2(\Gc_K^a) \times \wedge^2(\k_1(K)) \rightarrow \Z/\ell. \]
On the one hand, note that the commutator $[\bullet,\bullet]$ defined above yields a canonical map $[\bullet,\bullet] : \widehat\wedge^2(\Gc_K^a) \rightarrow \Gc_K^{(2)}/\Gc_K^{(3)} =: \Zf_K$, whose image is denoted by $\Zf_K^0$, as defined in \S\ref{subsec:mod-ell-analogue}.
On the other hand, the product-map on mod-$\ell$ Milnor K-theory yields a surjective map $\{\bullet,\bullet\} : \wedge^2(\k_1(K)) \rightarrow \k_2(K)$, since $-1 \in K^{\times \ell}$ and thus $\{x,x\} = \{x,-1\} = 0$ for all $x \in \k_1(K)$ (cf. \cite{Gille2006} Lemma 7.1.2).
The following theorem relates these two maps.
\begin{theorem}
\label{thm:gal-to-milnor}
Let $K$ be a field of characteristic $\neq \ell$ such that $\mu_{\ell^2} \subset K$.
Let $\omega$ be a primitive $\ell$-th root of unity in $K$.
Let $R$ denote the kernel of the map $[\bullet,\bullet] : \widehat\wedge^2(\Gc_K^a) \rightarrow \Gc_K^{(2)}/\Gc_K^{(3)}$.
Then the inclusion $R \hookrightarrow \widehat\wedge^2(\Gc_K^a)$ is $(\bullet,\bullet)_\omega$-dual to the (surjective) product map $\wedge^2(\k_1(K)) \rightarrow \k_2(K)$.
\end{theorem}
\begin{proof}
Note that the Merkurjev-Suslin Theorem \cite{Merkurjev1982} implies that the map $\cup \circ \inf : \H^1(\Gc_K^a) \otimes \H^1(\Gc_K^a) \rightarrow \H^2(\Gc_K)$ is surjective (see the discussion for Fact \ref{fact:H2-Galois-isom}).
This observation, combined with Lemma \ref{lemma:pre-image-of-H2-dec}, yields the following commutative diagram with exact rows:
\[\xymatrix{
0 \ar[r] & (\Gc_K^{(2)}/\Gc_K^{(3)})^\vee \ar[r]^{-d_2} \ar@{^(->}[d] & \H^2(\Gc_K^a)_{\rm dec} \ar[r]^{\inf} \ar@{^(->}[d] & \H^2(\Gc_K) \ar@{=}[d] \ar[r] & 0 \\
0 \ar[r] & \H^1(\Gc_K^{(2)})^{\Gc_K} \ar[r]_{-d_2} & \H^2(\Gc_K^a) \ar[r]_{\inf} & \H^2(\Gc_K) \ar[r] & 0 \\
}\]
Note that we have replaced $d_2$ with $-d_2$ in the diagram above in order to account for the negative sign that appears in the definition of $\Upsilon$ above.
As discussed above, we will identify $\H^2(\Gc_K^a)_{\rm dec}/(\beta \H^1(\Gc_K^a) \cap \H^2(\Gc_K^a)_{\rm dec}) = \wedge^2(\H^1(\Gc_K^a))$ using Fact \ref{fact:H2-presentation}, and we will consider the map induced by $-d_2$:
\[ \Upsilon : (\Gc_K^{(2)}/\Gc_K^{(3)})^\vee \xrightarrow{-d_2} \H^2(\Gc_K^a)_{\rm dec} \twoheadrightarrow \wedge^2(\H^1(\Gc_K^a)). \]

Since the map $\H^2(\Gc_K^a)_{\rm dec} \rightarrow \H^2(\Gc_K)$ factors through $\H^2(\Gc_K^a)_{\rm dec}/(\beta \H^1(\Gc_K^a) \cap \H^2(\Gc_K^a)_{\rm dec}) = \wedge^2(\H^1(\Gc_K^a))$ (see the discussion following Lemma \ref{lemma:bockstein-is-trivial-galois}),  we obtain an induced exact sequence:
\begin{align}
\label{align:cohom-ses}
(\Gc_K^{(2)}/\Gc_K^{(3)})^\vee \xrightarrow{\Upsilon} \wedge^2(\H^1(\Gc_K^a)) \xrightarrow{\cup \circ \inf} \H^2(\Gc_K) \rightarrow 0.	
\end{align}
By Fact \ref{fact: cup dual to commutator}, we see that this map $\Upsilon : (\Gc_K^{(2)}/\Gc_K^{(3)})^\vee \rightarrow \wedge^2(\H^1(\Gc_K^a))$ is precisely the dual of $[\bullet,\bullet] : \widehat\wedge^2(\Gc_K^a) \rightarrow \Gc_K^{(2)}/\Gc_K^{(3)}$.
In particular, the $\Z/\ell$-dual of (\ref{align:cohom-ses}) is none other than 
\[ 0 \rightarrow R \rightarrow \widehat\wedge^2(\Gc_K^a) \xrightarrow{[\bullet,\bullet]} \Gc_K^{(2)}/\Gc_K^{(3)}. \]
Thus, we have a canonical isomorphism $R^\vee \cong \H^2(\Gc_K)$ which is compatible with the corresponding maps from $\wedge^2(\H^1(\Gc_K^a))$ in the obvious sense.

To conclude the theorem, we simply recall that Kummer theory, along with our choice of $\omega$, induces an isomorphism $\k_1(K) \rightarrow \H^1(\Gc_K^a)$, which is given by $x \mapsto (\bullet,x)_\omega$.
We therefore obtain an induced isomorphism $\wedge^2(\k_1(K)) \rightarrow \wedge^2(\H^1(\Gc_K^a))$ which is exhibited by the fact that $\wedge^2(\k_1(K))$ is $(\bullet,\bullet)_\omega$-dual to $\widehat\wedge^2(\Gc_K^a)$.
Our choice of $\omega$ induces an isomorphism $\k_2(K) \rightarrow \H^2(\Gc_K)$ by the Merkurjev-Suslin theorem \cite{Merkurjev1982} and Fact \ref{fact:H2-Galois-isom}.
Furthermore, the multiplication maps are compatible in the sense that the following diagram is commutative:
\[
\xymatrix{
	\wedge^2(\k_1(K)) \ar[d]_{\cong} \ar@{->>}[r]^-{\{\bullet,\bullet\}} & \k_2(K) \ar[d]^{\cong} \\
	\wedge^2(\H^1(\Gc_K^a)) \ar@{->>}[r]_-{\cup \circ \inf} & \H^2(\Gc_K) \\
}
\]
Since $\wedge^2(\H^1(\Gc_K^a)) \twoheadrightarrow \H^2(\Gc_K)$ is dual to the inclusion $R \hookrightarrow \widehat\wedge^2(\Gc_K^a)$, we see that this inclusion $R \hookrightarrow \widehat\wedge^2(\Gc_K^a)$ is $(\bullet,\bullet)_\omega$-dual to the (surjective) multiplication map $\{\bullet,\bullet\} : \wedge^2(\k_1(K)) \twoheadrightarrow \k_2(K)$, as required.
\end{proof}

\section{Proof of Theorem \ref{thm:Main-Theorem-A}}
\label{sec:proof-of-theorem-galois}

We will use the notation of Theorem \ref{thm:Main-Theorem-A}.
Namely, $K|k$ and $L|l$ are function fields over algebraically closed fields such that $\Char k, \Char l \neq \ell$ and $\trdeg(K|k) \geq 5$.

\vskip 10pt
\noindent\emph{Proof of (1):}
Define $\kf_1 := \Hom(\Gc_K^a,\Z/\ell)$.
Let $R$ denote the kernel of $[\bullet,\bullet] : \widehat\wedge^2(\Gc_K^a) \rightarrow \Gc_K^{(2)}/\Gc_K^{(3)} = \Zf_K$ and define $\kf_2 := \Hom(R,\Z/\ell)$.
The dual of the (injective) map $R \hookrightarrow  \widehat\wedge^2(\Gc_K^a)$ is a surjective map
\[ \mu : \wedge^2(\kf_1) = \Hom(\widehat\wedge^2(\Gc_K^a),\Z/\ell) \twoheadrightarrow \Hom(R,\Z/\ell) = \kf_2. \]
Finally, for each quotient $\pi_H : \Gc_K^a \twoheadrightarrow \Gc_K^a/H$ in $\Rfrat(K|k)$, let $A_H$ denote the image of the dual map $\Hom(\Gc_K^a/H,\Z/\ell) \hookrightarrow \Hom(\Gc_K^a,\Z/\ell) = \kf_1$, and let $\Rfrat^\vee$ denote the collection $(A_H)_{\pi_H}$ of these subgroups of $\kf_1$, as $\pi_H$ varies over the elements of $\Rfrat(K|k)$.
To summarize, we have constructed the following data $\Kf$:
\begin{itemize}
	\item Two abelian groups $\kf_1$ and $\kf_2$. 
	\item A surjective homomorphism $\mu : \wedge^2(\kf_1) \twoheadrightarrow \kf_2$.
	\item A collection $\Rfrat^\vee$ of subgroups of $\kf_1$.
\end{itemize}

Let $\omega$ be a primitive $\ell$-th root of unity in $K$, which induces an isomorphism of $G_K$-modules $\mu_\ell \cong \Z/\ell$.
By Kummer theory, $\omega$ induces an isomorphism $\k_1(K) \cong \kf_1$.
Also, by Theorem \ref{thm:gal-to-milnor}, we see that $\omega$ induces an isomorphism $\k_2(K) \cong \kf_2$, and that the multiplication map $\k_1(K) \otimes \k_1(K) \twoheadrightarrow \k_2(K)$ corresponds to the composition $\kf_1 \otimes \kf_1 \twoheadrightarrow \wedge^2(\kf_1) \xrightarrow{\mu} \kf_2$ via these isomorphisms.
Finally, Kummer theory again shows that the isomorphism $\k_1(K) \cong \kf_1$ sends the elements of $\Gfor(K|k)$ bijectively onto $\Rfrat^\vee$.
By Theorem \ref{thm:Main-Theorem-A-milnor}(1), we know how to reconstruct $K^i|k$ from the data $\Kcalm(K|k)$, and thus the discussion above shows how to reconstruct $K^i|k$ from the data $\Kf$, hence from $\Gc_K^c$ together with $\Rfrat(K|k)$.

\vskip 10pt
\noindent\emph{Proof of (2):}
Assume that $\UIsom^c(\Gc_L^a,\Gc_K^a)$ is non-empty.
We first prove that $\Char L \neq \ell$.
Assume the contrary, and recall that this implies that $\Gc_L$ is a free pro-$\ell$ group.
Thus, one has $\H^2(\Gc_L) = 0$.
By (\ref{align:LHS}), Lemma \ref{lemma:pre-image-of-H2-dec} and Fact \ref{fact: cup dual to commutator}, this implies that the canonical map 
\[[\bullet,\bullet] : \widehat\wedge^2(\Gc_L^a) \rightarrow \Gc_L^{(2)}/\Gc_L^{(3)} \]
is \emph{injective}.
On the other hand, since $\UIsom^c(\Gc_L^a,\Gc_K^a)$ is non-empty, the same condition holds for $\Gc_K^a$.
Namely, the map $[\bullet,\bullet] : \widehat\wedge^2(\Gc_K^a) \rightarrow \Gc_K^{(2)}/\Gc_K^{(3)}$ is injective as well.
By Theorem \ref{thm:gal-to-milnor}, this implies that $\k_2(K) = 0$, which contradicts Lemma \ref{lemma:independent-coords}.
Hence $\Char L \neq \ell$.

To conclude, we must show that the canonical map 
\[ \Isom^i_F(K,L) \rightarrow \UIsom^c_{\rm rat}(\Gc_L^a,\Gc_K^a)\]
is a bijection.
Choose a primitive $\ell$-th root of unity $\omega_K$ in $K$ and a primitive $\ell$-th root of unity $\omega_L$ in $L$.
With these choices made and arguing as in the proof of (1) above, Theorem \ref{thm:gal-to-milnor} shows that the Kummer pairing yields a \emph{bijection}
\[\Psi_{\omega_K,\omega_L} : \Isomrat(\k_1(K),\k_1(L)) \rightarrow \Isom^c_{\rm rat}(\Gc_L^a,\Gc_K^a).\]
More precisely, for $\phi \in \Isomrat(\k_1(K),\k_1(L))$, the isomorphism $\Psi_{\omega_K,\omega_L}(\phi) : \Gc_L^a \rightarrow \Gc_K^a$ is defined explicitly by the condition that (using the notation of \S\ref{sub:Kummer-theory}):
\[ (\sigma,\phi(x))_{\omega_L} = (\Psi_{\omega_K,\omega_L}(\phi)(\sigma),x)_{\omega_K} \]
for all $\sigma \in \Gc_L^a$ and $x \in \k_1(K)$.
In the terminology of \S\ref{subsec:galois-variant} and \S\ref{subsec:milnor-variant}, it follows from Kummer theory that $\phi$ is compatible with $\Gfor$ if and only if $\Psi_{\omega_K,\omega_L}(\phi)$ is compatible with $\Rfrat$.
Also, in the terminology of \S\ref{subsec:mod-ell-analogue} and \S\ref{subsec:milnor-variant}, it follows from Theorem \ref{thm:gal-to-milnor} that $\phi$ is compatible with $\k_2$ if and only if $\Psi_{\omega_K,\omega_L}(\phi)$ is compatible with $[\bullet,\bullet]$.

Note that the map $\Psi_{\omega_K,\omega_L}$ depends on the choices of $\omega_K$ and $\omega_L$.
Nevertheless, changing $\omega_K$ and/or $\omega_L$ has the effect of replacing $\Psi_{\omega_K,\omega_L}$ by $\epsilon \cdot \Psi_{\omega_K,\omega_L}$ for some $\epsilon \in (\Z/\ell)^\times$.
Therefore, we obtain a \emph{canonical} bijection on $(\Z/\ell)^\times$-equivalence classes which is completely independent from choosing roots of unity:
\[ \UIsomm_{\rm rat}(\k_1(K),\k_1(L)) \xrightarrow{\cong} \UIsom^c_{\rm rat}(\Gc_L^a,\Gc_K^a).\]
Furthermore, by Kummer theory, this bijection is compatible with the canonical maps from $\Isom^i_F(K,L)$ in the sense that the following diagram commutes:
\[
\xymatrix{
{} & \UIsomm_{\rm rat}(\k_1(K),\k_1(L)) \ar[dd]^{\cong} \\
\Isom^i_F(K,L) \ar[ur]^{\rm canon.} \ar[dr]_{\rm canon.} \\
{} & \UIsom^c_{\rm rat}(\Gc_L^a,\Gc_K^a)
}\]
Thus, Theorem \ref{thm:Main-Theorem-A}(2) follows from Theorem \ref{thm:Main-Theorem-A-milnor}(2).

\end{document}